\documentclass[12pt]{amsart}
\pdfoutput=1 

\usepackage[
text={440pt,575pt},
headheight=9pt,
centering
]{geometry}

\usepackage{amsmath, amssymb, amsthm, enumerate, bm}
\allowdisplaybreaks 
\usepackage{mathtools}

\usepackage{graphicx}
\usepackage{subcaption}
\usepackage[normalem]{ulem}

\usepackage{multirow}

\usepackage{pifont}
\newcommand{\cmark}{\ding{51}}%
\newcommand{\xmark}{\ding{55}}

\usepackage{mathptmx}

\usepackage[cal=pxtx]{mathalfa}
\DeclareMathAlphabet{\pazocal}{OMS}{zplm}{m}{n}

\usepackage{xcolor}
\usepackage{hyperref}
\definecolor{darkblue}{RGB}{0,0,160}
\hypersetup{
colorlinks,%
citecolor=darkblue,%
filecolor=black,%
linkcolor=darkblue,%
urlcolor=darkblue
}

\usepackage{todonotes}
\makeatletter
\newcommand{\nolisttopbreak}{\vspace{\topsep}\nobreak\@afterheading}
\makeatother

\theoremstyle{definition}
\newtheorem{definition}{Definition}[section]

\newtheorem{theorem}[definition]{Theorem}
\newtheorem{proposition}[definition]{Proposition}
\newtheorem{lemma}[definition]{Lemma}
\newtheorem{corollary}[definition]{Corollary}

\newtheorem*{fact*}{Fact}
\newtheorem{remark}[definition]{Remark}
\newtheorem{example}[definition]{Example}

\newtheorem{problem}[definition]{Problem}

\renewcommand\>{\rangle}
\newcommand\<{\langle}

\newcommand{\A}{\mathcal{A}}

\newcommand{\C}{\mathcal{C}}

\newcommand{\M}{\mathcal{M}}
\newcommand{\NN}{\mathbb{N}}

\newcommand{\RR}{\mathbb{R}}
\newcommand{\R}{\mathcal{R}}
\renewcommand{\P}{\mathcal{P}}
\renewcommand{\S}{\mathcal{S}}

\newcommand{\ZZ}{\mathbb{Z}}

\def\Icc{{\mathcal I}}

\def\ol#1{{\overline {#1}}}

\newcommand{\lpi}{\makebox{\large\ensuremath{\varpi}}}

\newcommand{\compcent}[1]{\vcenter{\hbox{$#1\circ$}}}
\newcommand{\comp}{\mathbin{\mathchoice
  {\compcent\scriptstyle}{\compcent\scriptstyle}
  {\compcent\scriptscriptstyle}{\compcent\scriptscriptstyle}}}

\newcommand{\restr}[1]{|_{#1}}

\DeclareMathOperator{\cl}{cl}
\DeclareMathOperator{\ccl}{\operatorname{\text{\usefont{U}{BOONDOX-calo}{m}{n}cl}}}

\DeclareMathOperator{\idcl}{id}
\DeclareMathOperator{\cncl}{cn}
\DeclareMathOperator{\mndcl}{mn}

\DeclareMathOperator{\cid}{\operatorname{\text{\usefont{U}{BOONDOX-calo}{m}{n}i\kern-1.1pt{}d}}}
\DeclareMathOperator{\ccn}{\operatorname{\text{\usefont{U}{BOONDOX-calo}{m}{n}c\kern-0.3pt{}n}}}
\DeclareMathOperator{\cmnd}{\operatorname{\text{\usefont{U}{BOONDOX-calo}{m}{n}m\kern-1pt{}n}}}
\DeclareMathOperator{\cidl}{\operatorname{\text{\usefont{U}{BOONDOX-calo}{m}{n}i\kern-1.1pt{}d\kern-0.8pt{}l}}}

\DeclareMathOperator{\Sym}{Sym}
\DeclareMathOperator{\cSym}{\operatorname{\text{\usefont{U}{BOONDOX-calo}{m}{n}Sy\kern-1.5pt{}m}}}

\DeclareMathOperator{\Inc}{Inc}
\DeclareMathOperator{\cInc}{\operatorname{\text{\usefont{U}{BOONDOX-calo}{m}{n}In\kern-1.7pt{}c}}}

\DeclareMathOperator{\ini}{in}
\DeclareMathOperator{\Mon}{mon}
\DeclareMathOperator{\Mo}{Mon}

\DeclareMathOperator{\supp}{supp}

\DeclareMathOperator{\cone}{cone}

\DeclareMathOperator{\width}{width}

\DeclareMathOperator{\var}{var}

\DeclareMathOperator{\ind}{ind}

\newcommand{\oInc}{\overline{\Inc}}

\newcommand\defas{\coloneqq}

\title{
Invariant chains in algebra and discrete geometry
}

\author{Thomas Kahle, Dinh Van Le, and Tim R\"{o}mer}

\date{\today}


\thanks{The first author is partially supported by the DFG (314838170,
  GRK 2297, ``MathCoRe'').}

\begin{document}
\begin{abstract}
  We relate finite generation of cones, monoids, and ideals in
  increasing chains (the \emph{local} situation) to equivariant finite
  generation of the corresponding limit objects (the \emph{global}
  situation).  For cones and monoids there is no analog of
  Noetherianity as in the case of ideals and we demonstrate this in
  examples. As a remedy, we find local-global correspondences for
  finite generation.  These results are derived from a more general
  framework that relates finite generation under closure operations to
  equivariant finite generation under general families of maps.  We
  also give a new proof that non-saturated Inc-invariant chains of ideals stabilize, closing a gap in the literature.
\end{abstract}
\maketitle


\section{Introduction}
\label{sec:introduction}

Finite generation of algebraic and geometric objects is a central
necessity to efficiently work with these objects and to represent them
in a computer.  A well-known and important finiteness principle in
algebra is Noetherianity.  A commutative ring $R$ is Noetherian if
every ideal $I\subseteq R$ is finitely generated, or equivalently,
every ascending chain $I_{1}\subseteq I_{2}\subseteq \dotsb $ of
ideals eventually stabilizes, that is, from some index on, all
$\subseteq$ are equalities.

In some cases symmetry can augment finiteness.  E.g., the polynomial
ring $K[x_{1},\dotsc,x_{n}]$ over a field $K$ is Noetherian, but
$K[x_{1},x_{2},\dotsc]$ is not, since
$\<x_{1}\> \subsetneq \<x_{1},x_{2}\> \subsetneq \dotsb$ is an
infinite ascending chain.  Yet, polynomials
$f\in K[x_{1},x_{2},\dotsc]$ have finitely many terms, so each is
contained in some Noetherian subring
$K[x_{1},\dotsc, x_{n}] \subseteq K[x_{1},x_{2},\dotsc]$.  This
finiteness can be systematically investigated, for example, by
exploiting the action of symmetric groups, which renumber the
indeterminates.  The chain
$\<x_{1}\> \subsetneq \<x_{1},x_{2}\> \subsetneq \dotsb$ has the
property that its $n$-th ideal $\<x_{1},\dotsc,x_{n}\>$ arises from
the first $\<x_{1}\>$ by an action of the symmetric group $\Sym(n)$
and ideal closure:
$\<x_{1},\dotsc,x_{n}\> = \< \Sym(n) (\<x_{1}\>)\>$.  A~theorem of
Cohen~\cite{Co67} and Aschenbrenner-Hillar~\cite{AH07} states that if
a chain of ideals $(I_{n})_{n}$ with
$I_{n} \subseteq K[x_{1},\dotsc,x_{n}]$ is \emph{$\Sym$-invariant} in
the sense that $\<\Sym(n+k)(I_{n})\> \subseteq I_{n+k}$ for all
$n,k\in \NN$, then eventually the chain stabilizes, in the sense that
for large enough $n$, $\<\Sym(n+k)(I_{n})\> = I_{n+k}$ for all
$k\in \NN$.  Additionally, the union
$I_\infty \defas \bigcup I_{n} \subseteq K[x_{1},x_{2},\dotsc]$ is
generated by finitely many $\Sym$-orbits.  These facts, which are
called \emph{$\Sym$-Noetherianity} or \emph{equivariant Noetherianity}
of $K[x_{1},x_{2},\dotsc]$, have interesting
applications~\cite{AH07,Co67,hillar2012finite} and inspired lots of
recent work. See, e.g., \cite{Dr14,kahle2014equivariant,LNNR,LNNR2,
  le2020kruskal,nagel2017equivariant,nagel2019FIOI}.

We aim to explore such equivariant finiteness principles in the
broader context of discrete geometry.  For example, we are looking for
a framework to formulate equivariant versions of theorems from
polyhedral geometry.  For this one needs to consider chains of cones
$(C_{n})_{n}$ with $C_{n}\subseteq \RR^{n}_{\ge 0} $
with a suitable equivariance
(Definition~\ref{d:invariant} contains the general version).

Equivariant Noetherianity cannot hold in these setting as the ambient
$\RR^{n}_{\ge 0}$
has no Noetherianity:
There are non-polyhedral (i.e., not finitely generated) cones such as
the open orthant, defined by positivity of all coordinates.  A chain
of open orthants stabilizes but the limit is not finitely generated.
Examples~\ref{e:limitconevschain-beginning}--\ref{e:no-global-fg-Cone2} showcase these effects.

What remains of the theory is the equivalence of equivariant finite
generation in the limit (the \emph{global} situation) and
stabilization of the chain (the \emph{local} one) under appropriate assumptions.
For example, the \emph{local-global principle} for cones in
Corollary~\ref{c:globallocalcones} states that the limit $C_{\infty}$
of a $\Sym$-invariant family of cones is equivariantly finitely
generated if and only if the family $(C_{n})_{n}$ stabilizes and
eventually all $C_{n}$ are finitely generated.  This is also
equivalent to an eventual saturation condition
$C_{n} = C_{\infty} \cap \RR^{n}_{\ge 0}$ together with finite
generation of $C_{n}$ by rays of bounded support.
Section~\ref{section-cones-monoids} also contains a parallel
development for monoids.

The similarity between the results for cones and monoids points at a
generalization and unification, which we undertake in
Sections~\ref{sec:chain-set}--\ref{section-sets}.  In
Section~\ref{sec:chain-set} we abstract taking the ideal, cone, or
monoid to any closure operation and the action of $\Sym(n)$ to any
system of maps that maps objects of the chain into the later object.
The generalization has many advantages.  It allows to formulate a
general local-global principle, i.e.\ the exact conditions under which
the equivalence of finite generation up to symmetry and stabilization
hold.  This is our central Theorem~\ref{th:fg-and-stable}.
Specializing the maps to $\Sym$ and $\Inc$ (the monoid of increasing
maps) yields Theorems~\ref{thm:setFG} and~\ref{t:local-global-Inc}.
Further specializations to chains of polyhedral cones and monoids
under $\Sym$ and $\Inc$ follow in Section~\ref{section-cones-monoids}.
In Section~\ref{sec:ideals} we return to equivariant Noetherianity
in polynomial rings and use our results
to fill a gap in the proof that $\Inc$-invariant chains of ideals
stabilize, a fact that is used in the literature, e.g.\
in~\cite{NG18, nagel2017equivariant}.

Our work could be phrased in the framework of FI-modules by Church,
Ellenberg, and Farb (see, e.g., \cite{ChElFa14, FI, ChEl17}).  This
concerns, among other things, equivariant chains of modules using
symmetric groups.  The fundamental difference between our developments
in all but the last section is that the elements of our chains are
subsets of finite-dimensional vector spaces which rarely appear with
module structures, but rather take into account various closure
operations.  In the broader context of representation stability and
twisted commutative algebras, there is, in particular, the fundamental
work of Sam and Snowden (see, e.g., \cite{GL-sam-snowden, SaSn17}).
Unlike in our situation, Noetherianity abounds and is a central tool
in this theory.  In summary, while it would be possible to generalize
our results in Section~\ref{section-sets} and, based
on~\cite{nagel2019FIOI}, phrase them in the language of FI-modules we
chose a direct approach that is most suitable for applications like
the algebraic and geometric situations in
Sections~\ref{section-cones-monoids} and~\ref{sec:ideals}.

Finite generation is essential for computation.  To make our results
effective, more research is needed on the concrete stability indices,
for which very little is known.  On the side of polyhedral geometry,
there is software for dual description conversion modulo
symmetry~\cite{symmetric-poly}, but computer algebra for symmetric
ideals or monoids is in its infancy.

\section{Chains of sets}
\label{sec:chain-set}

We begin with a general framework of chains of sets and closure
operations.  Our goal is to be able to transfer properties from such
chains of sets to their union (or limit) and back.

Throughout this section $S_{\infty}$ is any set and $\S = (S_{n})_{n\ge 1}$
is an increasing chain of subsets
\begin{equation}
	\label{eq:chain}
	S_1\subseteq S_2\subseteq\cdots\subseteq S_n\subseteq\cdots\ \text{ such that } S_{\infty}=\bigcup_{n\ge 1} S_n.
\end{equation}
The set $S_{\infty}$ has its notation for consistency with the limits
below.  One could think of $S_{\infty}$ as an ambient set and $\S$ as
an ambient chain in which the chains of interest live.  A \emph{chain
	of sets} $\A=(A_{n})_{n\geq 1}$ with respect to $\S$ is an
increasing chain
\[
A_1\subseteq A_2\subseteq\cdots\subseteq A_n\subseteq\cdots
\]
with $A_n\subseteq S_n$ for all $n\ge 1.$ The \emph{limit set} of $\A$ is
\[
A_{\infty} = \bigcup_{n \geq 1} A_{n} \subseteq S_{\infty}.
\]
A chain of sets $\A$ is \emph{saturated} (respectively,
\emph{eventually saturated}) if
\[
A_n=A_{\infty}\cap S_n \ \text{ for all $n\ge1$ (respectively, for
	all $n\gg 0$)}.
\]
Here and in the following ``for all $n\gg 0$'' means that there exists
some $N\in\NN$ such that the property holds for all $n>N$.

Any chain $\A$ with limit set $A_{\infty}$ has a \emph{saturation}
that is a chain $\bar\A=(\ol A_{n})_{n\geq 1}$ defined by
\[
\ol A_{n} \defas A_{\infty}\cap S_n \ \text{ for all $n\ge1$}.
\]
Evidently, $\ol\A$ is the only saturated chain with limit
set~$A_{\infty}$.

The key objects of this paper are chains of sets that possess two
additional structures: First, each set is closed with respect to a
closure operation, and second, the chains are invariant under a group
or monoid action.  In the following we describe these two structures
as well as their compatibility.  Our notion of closure follows the
idea of E.H.\ Moore~\cite{moore1910} (although it differs slightly).
Let $\P(X)$ denote the power set of a set $X$.

\begin{definition}
	\label{d:sets-closure}
	A \emph{closure operation} on a set $X$ is a map
	$\cl\colon \P(X) \to \P(X)$ such that
	\begin{enumerate}
		\item
		$A\subseteq A^{\cl}$ for all $A\in\P(X)$.
		\item
		$A^{\cl}=(A^{\cl})^{\cl}$ for all $A\in\P(X)$.
		\item
		If $A,B\in \P(X)$ with $A\subseteq B$, then $A^{\cl} \subseteq B^{\cl}$.
	\end{enumerate}
	A set $A\in\P(X)$ is \emph{$\cl$-closed} if $A^{\cl}=A$.
\end{definition}

Let $\cl_\infty$ be a closure operation on $S_{\infty}$ and $\cl_n$ a
closure operation on $S_n$ for all $n\ge 1$.  We call $\cl_\infty$ a
\emph{global} closure operation and $\ccl=(\cl_n)_{n\geq 1}$ a chain
of \emph{local} closure operations.

\begin{definition}
	\label{d:consistency}
	A global closure operation $\cl_\infty$ is \emph{consistent} with a
	chain $\ccl$ of local closure operations (or $(\ccl, \cl_\infty)$ is
	a \emph{consistent system} of closure operations) if
	\[
	(A \cap S_n)^{\cl_n} = A^{\cl_{\infty}}\cap S_n \quad \text{ for all $n\ge 1$ and $A\subseteq S_{\infty}$}.
	\]
\end{definition}

\begin{remark}
	\label{r:local-consistency}
	Definition~\ref{d:consistency} formulates a \emph{local-global}
	consistency of the closure operations.  This implies the following
	\emph{local-local} consistency:
	\[
	(A_n \cap S_m)^{\cl_m} = A_n^{\cl_n}\cap S_m
	\quad \text{ for all $n\ge m\ge 1$ and $A_n\subseteq S_n$}.
	\]
	Indeed, it follows from Definition~\ref{d:consistency} that
	$A_n^{\cl_n}=(A_n \cap S_n)^{\cl_n} = A_n^{\cl_{\infty}}\cap S_n$.
	Thus,
	\[
	A_n^{\cl_n}\cap S_m=A_n^{\cl_{\infty}}\cap S_n\cap S_m=A_n^{\cl_{\infty}}\cap S_m=(A_n \cap S_m)^{\cl_m}.
	\]
	On the other hand, local-local consistency does not imply  local-global consistency. To see this, one can simply take $\cl_{\infty}$ to be the \emph{trivial} closure operation (i.e. $A^{\cl_{\infty}}=S_{\infty}$ for any $A\subseteq S_{\infty}$) and $\ccl$ the chain of identity closures considered in the next example.
\end{remark}

\begin{example}
	\label{e:closureop}
	The following are some closure operations that appear here.
	\begin{enumerate}[(i)]
		\item Letting $A^{\idcl_n}=A$ for $A\subseteq S_n$ yields the
		\emph{chain of identity closures} $(\idcl_n)_{n\geq 1}$.  Obviously, this chain satisfies the local-local consistency described above.
		\item Taking conical hulls $\cncl_n(A)=\cone(A)$ for $A\in\P(\RR^n)$
		induces the \emph{chain of conical closures}
		$\ccn=(\cncl_n)_{n\geq 1}$. See Section~\ref{section-cones-monoids} for details.
		\item Taking monoid closures $\mndcl_n(A)=\Mon(A)$
		for $A\in\P(\RR^n)$ induces the \emph{chain of monoid closures}
		$\cmnd=(\mndcl_n)_{n\geq 1}$. See Section~\ref{section-cones-monoids} for details.
		\item Let $K[X_{n}]$ be the polynomial ring from
		Section~\ref{sec:ambient-chain-examples}.  Taking ideal closures
		$\langle A \rangle_n$ for $A\in\P(K[X_n])$ induces the
		\emph{chain of ideal closures} $(\langle\cdot\rangle_n)_{n\geq 1}$. See Section~\ref{sec:ideals} for details.
	\end{enumerate}
\end{example}

In what follows we fix a global closure operation $\cl_\infty$ and a
chain $\ccl=(\cl_n)_{n\geq 1}$ of local closure operations.

\begin{definition}
	A chain of sets $\A=(A_{n})_{n\geq 1}$ is \emph{$\ccl$-closed}
	(respectively, \emph{eventually $\ccl$-closed}) if $A_n$ is
	$\cl_n$-closed for all $n\geq 1$ (respectively, for all $n\gg 0$).
\end{definition}

Given a $\ccl$-closed chain $\A$, one might ask whether the limit set
$A_\infty$ is $\cl_\infty$-closed, and vice versa, if the limit set
$A_\infty$ is $\cl_\infty$-closed, what can be said about the
chain~$\A$?  This type of local-global question is a frequent theme
here.  A partial but quite general answer is the following.

\begin{lemma}
	\label{l:closed-limit}
	Let $(\ccl, \cl_\infty)$ be a consistent system and
	$\A=(A_{n})_{n\geq 1}$ a $\ccl$-closed chain with limit
	set~$A_\infty$ and saturation $\ol\A$.  Then the following hold:
	\begin{enumerate}
		\item $A_\infty$ is $\cl_\infty$-closed if $\ol\A$ is
		eventually $\ccl$-closed or $\A$ is eventually saturated.
		\item If $A_\infty$ is $\cl_\infty$-closed, then $\ol\A$ is
		$\ccl$-closed.
	\end{enumerate}
\end{lemma}

\begin{proof}
	(i) Assume first that $\ol\A$ is eventually $\ccl$-closed. Then
	there exists $m\ge 1$ such that
	\[
	(A_\infty\cap S_n)^{\cl_n}= A_\infty\cap S_n\ \text{ for all } n\ge m.
	\]
	Since $S_\infty=\bigcup_{n\geq 1} S_n=\bigcup_{n\ge m} S_n$, it
	follows from Definition~\ref{d:consistency} that
	\begin{equation*}
		A_\infty^{\cl_{\infty}}
		=
		A_\infty^{\cl_{\infty}}\cap S_\infty
		=
		\bigcup_{n\ge m} (A_\infty^{\cl_{\infty}} \cap S_n)
		=
		\bigcup_{n\ge m} (A_\infty \cap S_n)^{\cl_n}
		=
		\bigcup_{n\ge m} (A_\infty \cap S_n)
		=
		A_\infty.
	\end{equation*}
	Hence, $A_\infty$ is $\cl_{\infty}$-closed.  If $\A$ is eventually
	saturated, then $A_\infty\cap S_n=A_n$ is $\cl_n$-closed for all
	$n\gg0$. Thus, the chain $\ol\A$ is eventually $\ccl$-closed and we
	conclude by the first case.

	(ii) If $A_\infty$ is $\cl_\infty$-closed, then by
	Definition~\ref{d:consistency},
	\begin{equation*}
		A_\infty\cap S_n
		= A_\infty^{\cl_{\infty}}\cap S_n
		= (A_\infty \cap S_n)^{\cl_n}
		\ \text{ for all $n\ge 1$}.
	\end{equation*}
	This means that the chain $\ol\A$ is $\ccl$-closed.
\end{proof}

The following lemma is immediate given Lemma~\ref{l:closed-limit}.
\begin{lemma}
	Under the assumptions of Lemma~\ref{l:closed-limit} the following
	are equivalent:
	\begin{enumerate}
		\item
		$A_\infty$ is $\cl_{\infty}$-closed;
		\item
		$\ol\A$ is eventually $\ccl$-closed;
		\item
		$\ol\A$ is $\ccl$-closed.
	\end{enumerate}
	Moreover, these equivalent statements hold if  $\A$ is eventually saturated.
\end{lemma}

Now that the chains are described, we focus on group or monoid actions
implementing symmetries within the chain.  In the abstract setting we
begin with just maps and specialize to actions later.  Let $\Pi$ be a
set of maps $S_{\infty}\to S_{\infty}$, and for $m\le n$ let
$\Pi_{m,n}$ be a set of maps $S_m \to S_n$.  We use the adjectives
\emph{global} to refer to the maps in $\Pi$ and \emph{local} for those
in the sets $\Pi_{m,n}$.  Let $\lpi$ denote the family
$\{\Pi_{m,n}\}_{m\le n}$. We call $(\lpi,\Pi)$ a \emph{system of
	maps}.  For $A_m\subseteq S_m$ and $A_{\infty}\subseteq S_{\infty}$
set
\begin{align*}
	\Pi_{m,n}(A_m)&=\{\pi(v)\mid \pi\in\Pi_{m,n},\ v\in A_m\}\subseteq S_n,\\
	\Pi(A_{\infty})&=\{\pi(v)\mid \pi\in\Pi,\ v\in A_{\infty}\}\subseteq S_{\infty}.
\end{align*}
In most cases of interest here $\Pi_{m,n}(A_m)$ is a finite set
whenever $A_m$ is a finite subset of~$S_m$.

\begin{definition}
	\label{d:locally-finite}
	A family $\lpi=\{\Pi_{m,n}\}_{m\le n}$ is \emph{locally
          finite} if for every $n\ge m\ge 1$ and every finite subset
        $A_m\subseteq S_m$ the set $\Pi_{m,n}(A_m)$ is finite.
\end{definition}

It is worth mentioning that a family $\lpi=\{\Pi_{m,n}\}_{m\le n}$ can
be locally finite even when each $\Pi_{m,n}$ is infinite; see
Lemma~\ref{l:Sym-Inc-locally-finite}.

Later, when $\Pi$ is a group or monoid, each set $\Pi_{m,n}$ can be
derived from $\Pi$, and moreover, the system $(\lpi,\Pi)$ is (weakly)
consistent in the sense of the next definition. In such cases, it is customary to use $\Pi$ as a representative for the family~$\lpi$.

\begin{definition}
	\label{d:Pi-consistency}
	A system of maps $(\lpi, \Pi)$ is
	\begin{enumerate}
		\item
		\emph{weakly consistent} if
		\begin{equation*}
			\Pi(A_{m})=\bigcup_{n\geq m}\Pi_{m,n}(A_{m})\ \text{ for all $m\ge 1$ and $A_m\subseteq S_m$};
		\end{equation*}
		\item
		\emph{consistent} if
		\begin{equation*}
			\Pi_{m,n}(A_m)=
			\Pi(A_{m})\cap S_n\ \text{ for all $n\ge m\ge 1$ and $A_m\subseteq S_m$}.
		\end{equation*}
	\end{enumerate}
\end{definition}

\vspace{8pt}  
\begin{remark}\label{r:consistency}\leavevmode
	\begin{enumerate}
		\item Consistency is stronger than weak consistency. Indeed, if
		$(\lpi,\Pi)$ is consistent, then for all $n\ge m\ge 1$ and
		$A_m\subseteq S_m$ one has
		\[
		\Pi(A_{m})
		=\Pi(A_{m})\cap S_\infty
		=\Pi(A_{m})\cap\Big(\bigcup_{n\geq m} S_n\Big)
		=\bigcup_{n\geq m}(\Pi(A_{m})\cap S_n)
		=\bigcup_{n\geq m}\Pi_{m,n}(A_{m}).
		\]
		\item Weak consistency is frequently applied in the form of the
		following easy consequence: For all $n\ge m$ and
		$A_{m} \subseteq S_{m}$
		it holds that $\Pi_{m,n} (A_{m}) \subseteq \Pi(A_{m})$.
		\item The consistency of $(\lpi, \Pi$) is a \emph{local-global}
		consistency.  It implies the following \emph{local-local}
		consistency:
		\[
		\Pi_{m,n}(A_{m})
		=\Pi_{k,n}(A_{m})\quad \text{ for all $n\ge k\ge m$ and
			$A_m\subseteq S_m\subseteq S_k$},
		\]
		since both sides are equal to $\Pi(A_{m})\cap S_n$.
	\end{enumerate}
\end{remark}

We are now ready to introduce the main object of study in this paper.

\begin{definition}
	\label{d:invariant}
	Let $(\lpi,\Pi)$ be a system of maps and let
	$\ccl=(\cl_n)_{n\geq 1}$ be a chain of closure operations.
	\begin{enumerate}
		\item
		A subset $A_{\infty}\subseteq S_{\infty}$ is \emph{$\Pi$-invariant} if
		\[
		\Pi(A_{\infty})\subseteq A_{\infty}.
		\]
		\item
		A $\ccl$-closed chain $\A=(A_{n})_{n\geq 1}$ is  \emph{$\lpi$-invariant} if
		\begin{equation*}
			(\Pi_{m,n} (A_{m}))^{\cl_n} \subseteq A_{n}
			\text{ whenever }  n\geq m.
		\end{equation*}
		A $\lpi$-invariant chain \emph{stabilizes} if there exists some
		integer $r\geq 1$ such that for all $n\geq m \geq r$ one has
		$(\Pi_{m,n} (A_{m}))^{\cl_n}= A_{n}$.  The smallest such $r$ is
		the \emph{$\lpi$-stability index} (or \emph{stability index}) of
		$\A$ with respect to $\ccl$ and denoted by
		$\ind^{\lpi}_{\ccl}(\A)$, or~$\ind_{\ccl}(\A)$, or even $\ind(\A)$
		if there is no danger of confusion.
	\end{enumerate}
\end{definition}

Let $\A=(A_{n})_{n\geq 1}$ be a chain of sets.  We would like to
describe properties of $\A$ (i.e.~\emph{local} properties) that can be
transferred to corresponding properties of the limit~$A_{\infty}$
(i.e.~\emph{global} properties) and vice versa.  After
Lemma~\ref{l:closed-limit}, here is the next example of such a
property.

\begin{lemma}
	\label{l:invariant}
	Let $(\lpi,\Pi)$ be a weakly consistent system of maps and
	$\A=(A_{n})_{n\geq 1}$ be a $\lpi$-invariant $\ccl$-closed chain
	with limit set~$A_{\infty}$.  Then the following hold:
	\begin{enumerate}
		\item $A_{\infty}$ is $\Pi$-invariant.
		\item If $A_{\infty}$ is $\cl_\infty$-closed, then $\ol\A$ is a
		$\lpi$-invariant $\ccl$-closed chain.
	\end{enumerate}
\end{lemma}

\begin{proof}
	(i) It follows from Definitions~\ref{d:Pi-consistency}(i) and
	\ref{d:invariant}(ii) that
	\begin{align*}
		\Pi(A_{\infty})
		&=\Pi\Big(\bigcup_{m \geq 1} A_{m}\Big)
		=\bigcup_{m \geq 1} \Pi(A_{m})
		=\bigcup_{n\ge m \geq 1} \Pi_{m,n}(A_{m})
		\subseteq \bigcup_{n \geq 1} A_{n}
		=A_{\infty}.
	\end{align*}
	Thus, $A_{\infty}$  is $\Pi$-invariant.
	
	(ii) By Lemma~\ref{l:closed-limit}(ii), $\ol\A$ is $\ccl$-closed. So
	it remains to show that $\ol\A$ is $\lpi$-invariant. Let $m,n\in\NN$
	with $n\ge m$. According to Remark~\ref{r:consistency}(ii) it
	holds that
	\[
	\Pi_{m,n}(A_\infty\cap S_{m})
	\subseteq \Pi(A_\infty\cap S_{m})
	\subseteq \Pi(A_\infty)
	\subseteq A_{\infty},
	\]
	where the last inclusion follows from (i). On the other hand, one has
	$\Pi_{m,n}(A_\infty\cap S_{m})\subseteq S_n$ by definition. Hence,
	$ \Pi_{m,n}(A_\infty\cap S_{m})\subseteq A_{\infty} \cap S_n, $ which
	implies
	\[
	(\Pi_{m,n}(A_\infty\cap S_{m}))^{\cl_n}\subseteq A_{\infty} \cap S_n
	\]
	since $\ol\A$ is $\ccl$-closed. Therefore,  $\ol\A$ is $\lpi$-invariant.
\end{proof}

The main focus of this work is to explore under which conditions local
finite generation implies global finite generation and vice versa.  We
first give a definition of this property in local and global
situations.

\begin{definition}
	\label{d:finite}
	Let $(\lpi,\Pi)$ be a system of maps and $(\ccl,\cl_\infty)$ a
	system of closure operations.  Let $\A=(A_{n})_{n\geq 1}$ be a
	$\ccl$-closed chain and $A\subseteq S_{\infty}$ a
	$\cl_\infty$-closed set.
	\begin{enumerate}
		\item (local) $\A$ is \emph{finitely generated} (respectively,
		\emph{eventually finitely generated}) if for all $n\ge1$
		(respectively, for all $n\gg0$) there is a finite subset
		$G_n\subseteq A_{n}$ such that $ A_n= G_n^{\cl_{n}}.$
		\item (global) $A$ is \emph{$\Pi$-equivariantly finitely generated}
		if there exists a finite subset $G\subseteq A$ such that
		$ A=\Pi(G)^{\cl_{\infty}}.$
	\end{enumerate}
\end{definition}

To discuss the relation between local and global finite generation we
introduce a crucial \emph{local-local compatibility} condition.
\begin{definition}
	\label{d:compatibility}
	A chain $\ccl=(\cl_n)_{n\ge 1}$ of closure operations and a family
	of maps $\lpi=\{\Pi_{m,n}\}_{m\le n}$ are \emph{compatible} if
	\[
	\Pi_{m,n}(A_m^{\cl_m}) \subseteq (\Pi_{m,n}(A_m))^{\cl_n}\ \text{ for all $n\geq m\ge 1$ and $A_m\subseteq S_m$}.
	\]
\end{definition}

\begin{lemma}
	\label{l:compatibility}
	If $\ccl$ and $\lpi$ are compatible, then
	\[
	(\Pi_{m,n}(A_m^{\cl_m}))^{\cl_n}=(\Pi_{m,n}(A_m))^{\cl_n}\ \text{ for all $n\geq m\ge 1$ and $A_m\subseteq S_m$}.
	\]
\end{lemma}

\begin{proof}
	From Definition~\ref{d:compatibility} it follows that
	\begin{align*}
		(\Pi_{m,n}(A_m^{\cl_m}))^{\cl_n}
		&\subseteq ((\Pi_{m,n}(A_m))^{\cl_n})^{\cl_n}
		=(\Pi_{m,n}(A_m))^{\cl_n}.
	\end{align*}
	The reverse inclusion is obvious since $A_m\subseteq A_m^{\cl_m}$.
\end{proof}

We are now ready to describe situations where the finite generation
property of a chain is inherited by its limit, and vice versa.

\begin{theorem}
	\label{th:fg-and-stable}
	Let $(\ccl,\cl_\infty)$ be a consistent system of closure operations
	and $(\lpi,\Pi)$ a system of maps.  Let $\A=(A_n)_{n\geq 1}$ be a
	$\lpi$-invariant, $\ccl$-closed chain with limit set~$A_\infty$.
	Consider the statements:
	\begin{enumerate}[(a)]
		\item (local) $\A$ stabilizes and is eventually finitely generated.
		\item (global) $A_\infty$ is $\Pi$-equivariantly finitely generated.
	\end{enumerate}
	The following hold:
	\begin{enumerate}
		\item If $A_\infty$ is $\cl_{\infty}$-closed, $(\lpi,\Pi)$ is weakly
		consistent, and $\ccl$ and $\lpi$ are compatible, then (a)
		implies~(b).
		\item If $\lpi$ is locally finite and $(\lpi,\Pi)$ is consistent, then
		(b) implies~(a).  Moreover, in this case $\A$ is eventually
		saturated.
	\end{enumerate}
\end{theorem}

\begin{proof}
	Assume first (a) and the assumptions of (i).  Since $\A$ stabilizes,
	its stability index $\ind(\A)$ is finite.  Since $\A$ is eventually
	finitely generated, we may choose an $m\geq \ind(\A)$ and a finite
	subset $G\subseteq A_{m}$ such that $A_m=G^{\cl_m}$.  So for $n\geq m$
	one obtains
	\begin{align*}
		A_n
		&=(\Pi_{m,n} (A_{m}))^{\cl_n}
		&&\text{(by Definition~\ref{d:invariant}(ii))}\\
		&=(\Pi_{m,n} (G^{\cl_m}))^{\cl_n}
		&&\text{(since $A_m=G^{\cl_m}$)}\\
		&=(\Pi_{m,n} (G))^{\cl_n}
		&&\text{(by Lemma~\ref{l:compatibility})}\\
		&=(\Pi_{m,n} (G)\cap S_n)^{\cl_n}
		&&\text{(since $\Pi_{m,n} (G)\subseteq S_n$)}\\
		&\subseteq (\Pi(G)\cap S_n)^{\cl_n}
		&&\text{(by Remark~\ref{r:consistency}(ii))}\\
		&=\Pi(G)^{\cl_{\infty}}\cap S_n
		&&\text{(by Definition~\ref{d:consistency})}.
	\end{align*}
	Since $A_n\subseteq A_m$ for $n\le m$ it follows that
	\begin{equation*}
		A_\infty
		=
		\bigcup_{n\geq 1} A_n
		=
		\bigcup_{n\geq m} A_n
		\;\subseteq\;
		\bigcup_{n\geq m} (\Pi(G)^{\cl_{\infty}}\cap S_n)
		=
		\Pi(G)^{\cl_{\infty}}.
	\end{equation*}
	On the other hand,
	$\Pi(G)^{\cl_{\infty}}\subseteq
	\Pi(A_\infty)^{\cl_{\infty}}\subseteq
	A_\infty^{\cl_{\infty}}=A_\infty$ by Lemma~\ref{l:invariant}(i)
	and the assumption that $A_\infty$ is $\cl_{\infty}$-closed.
	Hence, $A_\infty=\Pi(G)^{\cl_{\infty}}$ is $\Pi$-equivariantly
	finitely generated.
	
	\medskip Now assume (b) and the assumptions of (ii). Since
	$A_\infty$ is $\Pi$-equivariantly finitely generated, there exists
	a finite subset $G \subseteq A_\infty$ such that
	$A_\infty=\Pi(G)^{\cl_\infty}$.  Since $G$ is finite, we may
	assume $G\subseteq A_m$ for some large enough~$m$.  Thus, for
	$n\ge m$ one has
	\begin{align*}
		(\Pi_{m,n} (A_{m}))^{\cl_n}
		&\subseteq A_n\subseteq A_\infty\cap S_n\\
		&=\Pi(G)^{\cl_\infty}\cap S_n
		&&\text{(since $A_\infty=\Pi(G)^{\cl_\infty}$)}\\
		&=(\Pi(G)\cap S_n)^{\cl_n}
		&&\text{(by Definition~\ref{d:consistency})}\\
		&=(\Pi_{m,n} (G))^{\cl_n}
		&&\text{(by Definition~\ref{d:Pi-consistency}(ii))}\\
		&\subseteq (\Pi_{m,n} (A_{m}))^{\cl_n}
		&&\text{(since $G\subseteq A_m$)},
	\end{align*}
	hence equalities hold throughout.  That is,
	\[
	A_\infty\cap S_n=A_n=(\Pi_{m,n} (A_{m}))^{\cl_n}=(\Pi_{m,n} (G))^{\cl_n}.
	\]
	This shows that $\A$ stabilizes and is eventually saturated.
	Moreover, $\A$ is eventually finitely generated because $G$ is finite
	and $\lpi$ is locally finite.
\end{proof}

Now we formulate assumptions that guarantee equivalence of (a) and (b)
in Theorem~\ref{th:fg-and-stable}.

\begin{corollary}
	\label{c:fg-local-global}
	Let $(\ccl,\cl_\infty)$ and $(\lpi,\Pi)$ be consistent systems and
	$\lpi$ be locally finite and compatible with~$\ccl$.  If
	$\A=(A_n)_{n\geq 1}$ is a $\lpi$-invariant, $\ccl$-closed chain such
	that the limit $A_\infty$ is $\cl_{\infty}$-closed, then the following
	statements are equivalent:
	\begin{enumerate}[(a)]
		\item
		(local)
		$\A$ stabilizes and is eventually finitely generated;
		\item
		(global) $A_\infty$ is $\Pi$-equivariantly finitely generated.
	\end{enumerate}
	Moreover, if either of the above equivalent statements holds, then
	$\A$ is eventually saturated.
\end{corollary}

For later applications it is useful to relax the assumption that
$(\lpi,\Pi)$ is consistent in Theorem~\ref{th:fg-and-stable}(ii).
Analyzing the proof of this part, we see that its conclusion still
holds true if the consistency of $(\lpi,\Pi)$ is replaced by the
existence of a finite subset $G\subseteq A_\infty$ such that
\begin{equation}
	\label{eq:relaxation}
	A_\infty=\Pi(G)^{\cl_\infty}\quad \text{and}\quad
	\Pi(G)\cap S_n\subseteq \Pi_{m,n} (A_m)\ \text{ for all $n\ge m\gg0$}.
\end{equation}
Hence, we obtain the following.
\begin{proposition}
	\label{pr:relaxation}
	Let $(\ccl,\cl_\infty)$ be a consistent system of closure operations
	and $(\lpi,\Pi)$ be a system of maps.  Assume that $\lpi$ is locally
	finite.  Let $\A=(A_n)_{n\geq 1}$ be a $\lpi$-invariant,
	$\ccl$-closed chain with limit set~$A_\infty$.  If there exists a
	finite set $G\subseteq A_\infty$ such that
	condition~\eqref{eq:relaxation} is satisfied, then $\A$ stabilizes
	and is eventually finitely generated.
\end{proposition}

\section{Symmetric groups and a related monoid}
\label{sec:Sym and Inc}

We are mainly interested in systems of maps that are induced by actions of symmetric groups or the monoid of increasing functions. We discuss here some properties of these objects that will be used later.

Let $\NN = \{1,2,\dots \}$ denote the set of positive integers and,
for $n\in \NN$, set $[n]=\{1,\dots,n\}$.  We adopt the convention that
$[\infty]=\NN$.  Let $\Sym(n)$ denote the symmetric group on $[n]$ for
any $n\in \NN$.  Since $\Sym(n)$ can be naturally regarded as the
stabilizer subgroup of $n+1$ in $\Sym(n+1)$, we have an increasing
chain of finite symmetric groups
$
  \Sym(1)\subseteq \Sym(2)\subseteq\cdots\subseteq \Sym(n)\subseteq\cdots.
$
The limit of this chain is
\[
  \Sym(\infty)\defas\bigcup_{n\geq 1} \Sym(n).
\]
We often use $\Sym$
as an abbreviation for $\Sym(\infty)$. For $m\le n$ let
$\iota_{m,n}\colon [m]\to [n]$ denote the canonical embedding, i.e.\
$\iota_{m,n}(k)=k$ for all $k\in[m]$. We define
\[
  \Sym_{m,n} \defas \Sym(n)\comp\iota_{m,n}
  =\{\sigma\comp\iota_{m,n}\mid \sigma\in\Sym(n)\}.
\]
Let $\cSym$ denote the family $\{\Sym_{m,n}\}_{m\le n}$.
In the cases of interest, $(\cSym,\Sym)$ is consistent by Lemma~\ref{l:Sym-Inc-consistency}(ii).

Consider next the \emph{monoid of strictly increasing
  maps on $\NN$}, defined as
\[
  \Inc \defas \{ \pi \colon \NN \to \NN \mid  \pi(n)<\pi(n+1) \text{ for all } n\geq 1\}.
\]
When $m\le n$ we set
\[
  \Inc_{m,n} \defas \{\pi \in \Inc \mid \pi(m) \le n\}.
\]
Thus, each $\Inc_{m,n}$ is a subset of $\Inc$ and for any $m\ge 1$
there is an increasing chain
\begin{equation}
  \label{eq:Inc}
 \Inc_{m,m}\subseteq \Inc_{m,m+1}\subseteq\cdots\subseteq
 \Inc_{m,n}\subseteq\cdots \quad \text{with limit} \quad
  \Inc =\bigcup_{n\ge m}\Inc_{m,n}.
\end{equation}

Let $\cInc$ denote the family $\{\Inc_{m,n}\}_{m\le n}$. Then in the
setting of Section~\ref{section-sets}, $(\cInc,\Inc)$ is a weakly
consistent system of maps by Lemma~\ref{l:Sym-Inc-consistency}(i).
Since $\pi([m])\subseteq [n]$ for any $\pi\in\Inc_{m,n}$, we may view
the restriction $\pi\restr{[m]}$ as a map $[m] \to [n]$.  It is also
useful to consider the set
\[
\oInc_{m,n} \defas \{\ol\pi\colon [m]\to [n]\mid \text{ there exists } \pi \in \Inc_{m,n} \text{ with }\ol\pi=\pi\restr{[m]}\},
\]
which can be regarded as the quotient of $\Inc_{m,n}$ by the
equivalence relation: $\pi_1\sim \pi_2$ if
$\pi_1\restr{[m]}=\pi_2\restr{[m]}$.
%
The next lemma records a useful relationship between $\Sym$ and
$\Inc$.
\begin{lemma}
\label{l:Inc_and_Sym}
Let $m\in\NN$ and $n\in\NN\cup\{\infty\}$ with $m\le n$. Then the
following hold:
\begin{enumerate}
\item $\Sym_{m,n}$ equals the set of injective maps from
  $[m] \to [n]$ and $\oInc_{m,n}$ equals the set of injective 
  order-preserving maps $[m] \to [n]$.
\item
 $\oInc_{m,n}\subseteq \Sym_{m,n}.$
 In particular, $\pi\restr{[m]}\in \Sym_{m,n}$ for every $\pi\in \Inc_{m,n}$.
 \item
 $\Sym_{m,n}
 =\oInc_{m,n}\comp\Sym(m):=
 \{\pi\comp\sigma\mid \pi\in \oInc_{m,n},\sigma\in\Sym(m)\}.$
\end{enumerate}
\end{lemma}

\begin{proof}
  (i) and (ii) are easy to check directly from the definitions.
  Using (i), a proof of (iii) amounts to checking that each injective map
  $[m]\to[n]$ factors into a reordering of $[m]$ followed by an 
  order-preserving injective map $[m] \to [n]$.
\end{proof}

We also need the following decomposition, which can be shown to also
hold for~$\Sym$.

\begin{lemma}
 \label{l:Inc}
 For any $m,n\in\NN$ with $n>m$ it holds that
 \[
   \Inc_{m,n}=\Inc_{m+1,n}\comp \Inc_{m,m+1}=\Inc_{n-1,n}\comp \Inc_{m,n-1}.
 \]
\end{lemma}

\begin{proof}
  From \cite[Proposition~4.6]{nagel2017equivariant} it follows that
  $
    \Inc_{m,n}=\Inc_{m+1,n}\comp \Inc_{m,m+1}.
  $
  Based on this decomposition, one can show by induction that
  \begin{equation*}
    \Inc_{m,n}
    =\Inc_{n-1,n}\comp\Inc_{n-2,n-1}\comp\dotsb\comp \Inc_{m,m+1} = \Inc_{n-1,n} \comp \Inc_{m,n-1}. \qedhere
  \end{equation*}
\end{proof}

\section{Sets up to symmetry}
\label{section-sets}
In this section we specialize the theory in
Section~\ref{sec:chain-set} to $\Sym$- and $\Inc$-invariant chains.
For simplicity, we restrict the discussion to a setup which is
compatible with the ones in the forthcoming sections so that the
results in this section can be applied in a direct and effective way.
To this end, we only consider chains of sets with respect to an
ambient chain $\R=(R_{n})_{n\geq 1}$ in which each $R_n$ is either a
polynomial ring or a subset of~$\RR^n$.

\subsection{Ambient chains}
\label{sec:ambient-chain-examples}

We describe two types of ambient chains $\R=(R_{n})_{n\geq 1}$ for the
rest of this paper. The first is a chain of polynomial rings.  Let $K$
be a field and $c$ a positive integer.  For any $n\in \NN$ let $R_n$
be the polynomial ring
$K[X_n]\defas K[x_{i,j}\mid i\in [c],\ j\in [n]]$.

Then, for $m \le n$, the embedding $\iota_{m,n}\colon [m]\to [n]$
induces a canonical embedding $R_{m}\to R_{n}$, also denoted
by~$\iota_{m,n}$, that maps each polynomial in $R_{m}$ to the same
polynomial considered as an element of $R_{n}$.  Thus, we obtain the
chain
\[
  R_1\subseteq R_2\subseteq\cdots\subseteq R_n\subseteq\cdots
\text{ with limit }
  R_{\infty}= \bigcup_{n\geq 1} R_{n} = K[X]\defas K[x_{i,j}\mid i\in [c],\
  j\in \NN].
\]

The second type of the ambient chain $\R=(R_{n})_{n\geq 1}$ are chains
with $R_n\subseteq \RR^n$ for all $n\geq 1$ that are
\emph{$\cSym$-invariant}, meaning that
\begin{equation}
\label{eq:inclusion}
 \Sym_{m,n}(R_m)\subseteq R_n\ \text{ for all } n\ge m\ge 1.
\end{equation}
The inclusion is to be understood as follows.  Every element
$\pi=\sigma\comp\iota_{m,n}\in \Sym_{m,n}$ with $\sigma\in \Sym(n)$
gives rise to a map $\RR^m\to\RR^n$.  First, the embedding
$\iota_{m,n}\colon [m]\to [n]$ induces a canonical inclusion
$\iota_{m,n}\colon\RR^m\to\RR^n$ that embeds $\RR^{m}$ as the first
$m$ coordinates in $\RR^{n}$:
\begin{equation}
\label{eq:iota}
\iota_{m,n}(v)=(v,0,\dotsc,0)\in \RR^{n}\ \text{ for any } v\in \RR^{m}.
\end{equation}
Then $\sigma\in \Sym(n)$ acts on $\RR^n$ by permuting coordinates,
that is
\begin{equation}
\label{eq:sigma}
 \sigma (v_1,\dots, v_n)=(v_{\sigma^{-1}(1)},\dots,v_{\sigma^{-1}(n)})\quad\text{for any } (v_1,\dots, v_n)\in \RR^{n}.
\end{equation}
Thus, each $\pi=\sigma\comp\iota_{m,n}\in \Sym_{m,n}$ defines a map
$\pi\colon\RR^m\to\RR^n$, and we understand~\eqref{eq:inclusion} as
\[
 \Sym_{m,n}(R_m)\defas\{\pi(R_m)\mid \pi\in \Sym_{m,n}\}\subseteq R_n\ \text{ for all } n\ge m\ge 1.
\]
Via the embedding
$\iota_{m,n}$ we regard $R_m$ as a subset of $R_n$ for $m\le n$
and thereby get a chain
with limit $R_{\infty}= \bigcup_{n\geq 1} R_{n}$ contained in the infinite dimensional vector space
$\RR^{(\NN)}\defas \bigcup_{n\geq 1} \RR^{n}$.
The canonical basis of $\RR^{(\NN)}$ consists of the vectors
$\epsilon_{i}$, $i\in\NN$, with
\[
  (\epsilon_{i})_{j} =
  \begin{cases}
    1 & \text {if $i = j$},\\
    0 & \text{otherwise}.
  \end{cases}
\]

With respect to the identity closures in Example~\ref{e:closureop}(i), an
ambient chain $\R=(R_{n})_{n\geq 1}$ of the second type is exactly a
$\cSym$-invariant chain in the sense of
Definition~\ref{d:invariant}(ii).  Chains of this type include the
cases where $R_n$ equals $\RR^n$, $\RR_{\ge0}^n$, $\ZZ^n$, or
$\NN^n$ for all $n\ge 1$.

From now on, we always assume the ambient
chain $\R=(R_{n})_{n\geq 1}$ is one of the two types described above.
For $m\in\NN$
the canonical embedding $\iota_{m,\infty}\colon R_{m}\to R_{\infty}$ allows to identify $R_m$ with a subset
of $R_{\infty}$.  Thus, for any $v\in R_m$, the
elements $v$, $\iota_{m,n}(v)$, $\iota_{m,\infty}(v)$ are all
identical for every $n\ge m$.
By definition, each $w\in R_{\infty}$ is contained in some $R_{m}$. The
smallest such $m$ is the \emph{width} of~$w$:
\begin{equation*}
 \width(w)\defas\min\{m\in\NN\mid w\in R_m\}.
\end{equation*}
When $R_{\infty}= K[X]$, let $\var(w)$ denote the set of variables appearing in~$w$.
If
$R_{\infty}\subseteq\RR^{(\NN)}$, then we write $w=(w_1,w_2,\dots)$ with
$w_j\in \RR$. The \emph{support} of $w$ is
\[
 \supp(w)\defas
 \begin{cases}
  \{j\in\NN\mid x_{i,j}\in\var(w)\text{ for some } i\in[c]\}&\text{if } R_{\infty}= K[X],\\
  \{j\in\NN\mid w_j\ne0\}&\text{if } R_{\infty}\subseteq\RR^{(\NN)}.
 \end{cases}
\]
In the polynomial ring case our definition of support is coarser than
the usual definition, in which $\supp(w)$ is the set of monomials
of~$w$.  However, the above definition is more useful for our
purposes.  We call $|\supp(w)|$ the \emph{support size} of $w$. Evidently, $|\supp(w)|\le \width(w)$.

\subsection{Sym and Inc actions}

Given an ambient chain $\R=(R_{n})_{n\geq 1}$, the systems
$(\cSym,\Sym)$ and $(\cInc,\Inc)$ give rise to systems of maps on
$(\R,R_{\infty})$.  In particular, this yields actions of
$\Sym(\infty)$ and $\Inc$ on the ambient set~$R_{\infty}$.  We now
collect basic properties of these actions and discuss the consistency
and local finiteness of $(\cSym,\Sym)$ and $(\cInc,\Inc)$.

Consider $(\cSym,\Sym)$ first.  Let $m\le n$ and
$\pi=\sigma\comp\iota_{m,n}\in \Sym_{m,n}$ with $\sigma\in \Sym(n)$.
Then $\pi$ induces a map
$\pi\colon R_m\to R_n$ given by
\begin{equation}
	\label{eq:pi}
\pi(v)=
\begin{cases}
	\sum_{k\ge1} v_{k}\epsilon_{\pi(k)} &\text{if } v = \sum_{k\ge 1}v_{k}\epsilon_{k}\in R_m\subseteq \RR^{m},\\
	x_{i,\pi(j)} &\text{if } v = x_{i,j}\in R_m=K[X_m],
\end{cases}
\end{equation}
where in the first case, $\pi$ is precisely the map induced by the composition of the maps
in~\eqref{eq:iota} and~\eqref{eq:sigma}, which is well-defined because the chain $\R$ is $\cSym$-invariant. By this definition, $\Sym_{m,n}$ is identified with a set of maps $R_m \to R_n$.  In particular, the group $\Sym_{n,n} = \Sym(n)$ acts on~$R_n$, and moreover, the
action of $\Sym(n+1)$ on $R_{n+1}$ extends that of $\Sym(n)$ on~$R_n$
for all $n\ge 1$.  Since $\Sym(\infty)=\bigcup_{n\ge 1}\Sym(n)$ and
$R_\infty=\bigcup_{n\ge 1}R_n$, these actions together build an action
of $\Sym(\infty)$ on~$R_{\infty}$.  We have thus defined the system of maps $(\cSym,\Sym)$ on $(\R,R_{\infty})$.

Some useful observations from the above definition are recorded in the next lemma.

\begin{lemma}
  \label{l:Sym-Sym}
 Let $m,n\in\NN$ with $m\le n$ and $v\in R_m$.  Then the following hold:
 \begin{enumerate}
  \item
  $|\supp(v)|=|\supp(\pi(v))|$ for every $\pi\in\Sym_{m,n}.$
  \item
  $\Sym_{m,n}(v)=\Sym(n)(v).$
 \end{enumerate}
\end{lemma}

\begin{proof}
 (i) follows easily from~\eqref{eq:pi}.  To see (ii),  recall that $v$ and $\iota_{m,n}(v)$ are identified in~$R_n$. Hence,
  $
    \Sym_{m,n}(v)
    =(\Sym(n)\comp\iota_{m,n})(v)
    =\Sym(n)(\iota_{m,n}(v))
    =\Sym(n)(v).
 $
\end{proof}

Next, we construct the system of maps $(\cInc,\Inc)$, based on the system $(\cSym,\Sym)$.
Let $\pi\in\Inc_{m,n}$.  Then
$\ol\pi\defas\pi\restr{[m]}\in\oInc_{m,n}\subseteq\Sym_{m,n}$ by Lemma~\ref{l:Inc_and_Sym}(ii).  Hence,
$\ol\pi$ induces a map $R_m \to R_n$, explicitly described in~\eqref{eq:pi}.  So we can define a map
$\pi\colon R_m\to R_n$
by letting $\pi(v)=\ol\pi(v)$ for all $v\in R_m$.  Now let $\pi'\in \Inc$ and
$w\in R_\infty$.  Then $w\in R_{m}$ for some $m$ and there exists
$n\ge m$ such that $\pi'\in \Inc_{m,n}$.  Evidently, $\pi'(w)$ does not depend on the choice of $m$ and $n$.  This implies that $\pi'$ induces a map $\pi'\colon R_\infty\to R_\infty$.  Thus, we obtain the system of maps $(\cInc,\Inc)$ on $(\R,R_{\infty})$.

From the construction it is clear that \eqref{eq:pi} still holds if $\Sym$ is replaced by~$\Inc$.  In particular, $\Inc$ acts as a monoid on $R_{\infty}$. The next result compares the actions of $\Sym$ and $\Inc$.

\begin{lemma}
\label{l:orbits-inclusion}
Let $m,n\in\NN$ with $m\le n$. Then for any
$v\in R_m$ and $w\in R_\infty$ the following hold:
\begin{enumerate}
\item
$
\Inc_{m,n} (v) =\oInc_{m,n} (v)\subseteq \Sym_{m,n} (v).
$
\item
$
\Inc (w) \subseteq \Sym (w).
$
\end{enumerate}
\end{lemma}

\begin{proof}
  (i) follows from viewing $\pi\in\Inc_{m,n}$ as
  $\pi|_{[m]}\in\ol\Inc_{m,n}$.  To prove (ii), let $\pi\in \Inc$.  Then $w\in R_{m}$ and $\pi\in \Inc_{m,n}$ for some $n\ge m$.  So it follows from (i) and
Lemma~\ref{l:Sym-Sym}(ii) that
\begin{equation*}
 \pi(w)\in\Sym_{m,n}(w)=\Sym(n)(w)
 \subseteq \Sym(w).
 \qedhere
\end{equation*}
\end{proof}

Sometimes it is necessary to describe the truncated orbits
$\Sym(w)\cap R_n$ and $\Inc(w) \cap R_{n}$ of an element
$w\in R_\infty$.  Truncating at the width suffices to not lose
information.

\begin{lemma}
 \label{l:truncated-orbit}
 Let $w\in R_\infty$ and $m,n\in\NN$ with $n\ge m$.
 \begin{enumerate}
 \item If $\width(w)=m$, then $\Inc(w)\cap R_n=\Inc_{m,n}(w)$.
 \item If $w\in R_m$ (i.e.\ $\width(w)\le m$), then
   $\Sym(w)\cap R_n=\Sym(n)(w)=\Sym_{m,n}(w)$.
 \end{enumerate}
\end{lemma}

\begin{proof}
  (i) We may assume that $w\ne 0$. Let $n\ge m$. It suffices to prove
  the inclusion
  \[
    \Inc(w)\cap R_n\subseteq\Inc_{m,n}(w).
  \]
  Take $u\in\Inc(w)\cap R_n$. Then $\width(u)\le n$ and $u=\pi(w)$ for
  some $\pi\in\Inc$.

  Consider first the case that $R_{\infty}\subseteq \RR^{(\NN)}$.
  Using $\width(w)=m$, we write $w = \sum_{k= 1}^mw_{k}\epsilon_{k}$
  with~$w_m\ne0$.
  Then
  $
    u=\pi(w)=\sum_{k= 1}^mw_{k}\epsilon_{\pi(k)}.
  $
  This implies that $\pi(m)\le n$ since $w_m\ne0$ and $\width(u)\le n$. Hence,
  $\pi\in \Inc_{m,n}$ and $ u=\pi(w)\in \Inc_{m,n}(v)$.

  In the case $R_{\infty}=K[X]$, $w$ involves a variable $x_{i,m}$ for some $i\in[c]$ since $\width(w)=m$.  Thus $u=\pi(w)$ involves $x_{i,\pi(m)}$.  Since
  $\width(u)\le n$, one again gets $\pi\in \Inc_{m,n}$ as above.

  (ii) The first equality is easy to see, while the second follows
  from Lemma~\ref{l:Sym-Sym}(ii).
\end{proof}

We are now ready to discuss the consistency and local finiteness of the systems of maps
$(\cSym,\Sym)$ and $(\cInc,\Inc)$.

\begin{lemma}
 \label{l:Sym-Inc-consistency}
 \leavevmode\nolisttopbreak
 \begin{enumerate}
  \item
  The system $(\cInc,\Inc)$ is weakly consistent.
  \item
  The system $(\cSym,\Sym)$ is consistent.
 \end{enumerate}
\end{lemma}

\begin{proof}
  (i) follows from \eqref{eq:Inc} and (ii) from
  Lemma~\ref{l:truncated-orbit}(ii).
\end{proof}

The system $(\cInc,\Inc)$ is not consistent as
Example~\ref{ex:Inc-not-stable} shows.

\medskip Because of the preceding result and the way that the family
$\cSym$ (respectively, $\cInc$) is derived from $\Sym$ (respectively,
$\Inc$), one usually uses $\Sym$ (respectively, $\Inc$) as a
representative for $\cSym$ (respectively, $\cInc$).  So, for example,
the statement of the next result is another way of saying that the
families $\cSym$ and $\cInc$ are locally finite in the sense of
Definition~\ref{d:locally-finite}.  For simplicity we also use the
terms $\Sym$- or $\Inc$-\emph{invariant chain} instead of $\cSym$- or
$\cInc$-\emph{invariant chain}, and so forth.

\begin{lemma}
 \label{l:Sym-Inc-locally-finite}
 $\Sym$ and $\Inc$ are locally finite.
\end{lemma}

\begin{proof}
  Let $n\ge m\ge 1$ and $A_m\subseteq R_m$ be a finite subset. We have
  to show that $\Sym_{m,n}(A_m)$ and $\Inc_{m,n}(A_m)$ are finite
  sets. The first set is finite because $\Sym_{m,n}$ finite. The
  second one is also finite since $\Inc_{m,n}(A_m)=\oInc_{m,n}(A_m)$
  by Lemma~\ref{l:orbits-inclusion}(i) and $\oInc_{m,n}$ is finite.
\end{proof}

\subsection{Sym- and Inc-invariant chains}

As a consequence of Lemma~\ref{l:orbits-inclusion} there are more
$\Inc$-invariant chains than $\Sym$-invariant chains.
Generalizing \cite[Lemma~7.6]{nagel2017equivariant} we have:

\begin{lemma}
  \label{l:orbits}
  Let $\ccl=(\cl_n)_{n\geq 1}$ be a chain of closure operations.
  \begin{enumerate}
  \item If a $\ccl$-closed chain $(A_{n})_{n\geq 1}$ is
    $\Sym$-invariant, then it is $\Inc$-invariant.
  \item If $A\subseteq R_\infty$ is $\Sym$-invariant, then it is
    $\Inc$-invariant.
  \end{enumerate}
\end{lemma}

\begin{example} The converses of Lemma~\ref{l:orbits}(i) and (ii) are
  not true.  Consider the chain of identity closures as in
  Example~\ref{e:closureop}(i).  Let $A_{1} = \{0\}$ and
  $ A_{n} = \{0\} \times \RR^{n-1}_{\ge 0} \subseteq \RR^n\ \text{ for
  } n\geq 2$.  Then $\A=(A_n)_{n\geq 1}$ is an $\Inc$-invariant chain
  of sets.  However, it is not $\Sym$-invariant since, e.g.,
  $\Sym_{2,n} (A_{2})$ contains vectors $v\in \RR^n$ with $v_1\neq 0$.
\end{example}

\begin{remark}\label{r:chaff}\
  \begin{enumerate}[(i)]
  \item By Lemma~\ref{l:orbits}(i), one can view a $\Sym$-invariant
    chain as an $\Inc$-invariant chain.  This was used for chains of
    ideals in \cite{hillar2012finite} and elsewhere.  Also in our
    general setup, with some more effort, it can be shown that if
    $\A=(A_{n})_{n\geq 1}$ is a $\Sym$-invariant, $\ccl$-closed chain
    of sets, then $\A$ stabilizes as a $\Sym$-invariant chain if and
    only if it stabilizes as an $\Inc$-invariant chain.  In this case
    $\ind^{\Sym}(\A)=\ind^{\Inc}(\A)$.

  \item By Lemma~\ref{l:orbits}(ii), any $\Sym$-invariant subset
    $A\subseteq R_\infty$ is also $\Inc$-invariant.  Then it can be
    shown with some more effort that equivariant finite generation of
    a $\Sym$-invariant, $\cl_\infty$-closed set
    $A \subseteq R_{\infty}$ holds for $\Sym$ if and only if it holds
    for~$\Inc$.
\end{enumerate}
\end{remark}

The following characterization of stabilization, which generalizes \cite[Lemma~5.2]{nagel2017equivariant},
could be of independent interest.

\begin{proposition}
  \label{l:stabilitiycharacterization}
  Let $\ccl=(\cl_n)_{n\geq 1}$ be a chain of closure operations that
  is compatible with~$\Inc$.  Let $\A=(A_n)_{n\geq 1}$ be an
  $\Inc$-invariant, $\ccl$-closed chain of sets. Then for $r\in \NN$
  the following statements are equivalent:
\begin{enumerate}
\item $\A$ stabilizes and its stability index is at most $r$;
\item $\Inc_{n,n+1}(A_n)^{\cl_{n+1}} = A_{n+1}$ whenever $n\ge r$;
\item $\Inc_{r,n}(A_r)^{\cl_n} = A_{n}$ whenever $n\ge r$.
\end{enumerate}
\end{proposition}
\begin{proof}
  The implication (i) $\Rightarrow$ (ii) follows directly from the
  definition.  We prove (iii) from (ii) by induction on~$n$. First
  consider the case $n=r$. Since $\Inc_{r,r}$ contains the identity
  map, one has
  $ A_r \subseteq \Inc_{r,r}(A_r) \subseteq \Inc_{r,r}(A_r)^{\cl_r}
  \subseteq A_r$, hence $A_r=\Inc_{r,r}(A_r)^{\cl_r}$.

 Assume we have shown that $A_n=\Inc_{r,n}(A_r)^{\cl_n}$ for some
 $n\ge r$. Then $A_{n+1}$ equals
\begin{equation*}
  \Inc_{n,n+1}(A_n)^{\cl_{n+1}}
  =\Inc_{n,n+1}(\Inc_{r,n}(A_r)^{\cl_n})^{\cl_{n+1}}
  =\Inc_{n,n+1}(\Inc_{r,n}(A_r))^{\cl_{n+1}}
  =\Inc_{r,n+1}(A_r)^{\cl_{n+1}},
\end{equation*}
by (ii), the induction hypothesis, Lemma~\ref{l:compatibility}, and
Lemma~\ref{l:Inc}.

To prove (iii) $\Rightarrow$ (i), take any $n>r$.  We show that
$A_n=(\Inc_{m,n}(A_{m}))^{\cl_n}$ for all $n\geq m \geq r$ by
induction on $m$. The case $m=r$ follows from (iii). Assume
$A_n=(\Inc_{m,n}(A_{m}))^{\cl_n}$ for some $m$ with $n> m \geq r$. By
Lemma~\ref{l:Inc}, $\Inc_{m,n}=\Inc_{m+1,n}\comp \Inc_{m,m+1}$. It
follows that
\[
A_{n}
=
(\Inc_{m,n} (A_{m}))^{\cl_n}
=
(\Inc_{m+1,n}\comp \Inc_{m,m+1} (A_{m}))^{\cl_n}
\subseteq
(\Inc_{m+1,n}(A_{m+1}))^{\cl_n}
\subseteq
A_n.
\]
Hence, $A_n = (\Inc_{m+1,n}(A_{m+1}))^{\cl_n}$, which
concludes the induction argument.
\end{proof}

\subsection{Finite generation up to symmetry}
\label{section-glfg}

We apply and refine the results on local-global finite generation in
Section~\ref{sec:chain-set} to $\Sym$- and $\Inc$-invariant chains.
For this we consider a general system of closure operations
$(\ccl,\cl_\infty)$, leaving more specific discussions until later
sections.

\medskip For $\Sym$-invariant chains, apart from the local-global
principle in Corollary~\ref{c:fg-local-global}, there is one
additional local characterization.  The following result generalizes
\cite[Theorem~4.7]{AH07} and \cite[Corollary~3.7]{hillar2012finite};
see Section~\ref{sec:ideals} for more details.
\begin{theorem}
\label{thm:setFG}
Let $(\ccl,\cl_\infty)$ be a consistent system of closure operations
that is compatible with~$\Sym$.  Let $\A=(A_n)_{n\geq 1}$ be a
$\Sym$-invariant, $\ccl$-closed chain for which $A_\infty$ is
$\cl_{\infty}$-closed. Then the following statements are equivalent:
\begin{enumerate}[(a)]
\item
  (local)
  $\A$ stabilizes and is eventually finitely generated;
\item
  (local)
  There exists an $r\in\NN$ such that for all $n\ge r$ the following hold:
\begin{enumerate}[(i)]
\item
  (saturation)
  $A_\infty \cap R_{n} = A_{n}$,
\item
  (support size)
  $A_{n}$ is finitely generated by elements of support size at most $r$;
\end{enumerate}
\item
  (global) $A_\infty$ is $\Sym$-equivariantly finitely generated.
\end{enumerate}
\end{theorem}

\begin{proof}
  From Lemmas~\ref{l:Sym-Inc-consistency}(ii) and
  \ref{l:Sym-Inc-locally-finite} we know that the system of maps
  $(\cSym,\Sym)$ satisfies the remaining assumptions of
  Corollary~\ref{c:fg-local-global}.  So (a) and (c) are equivalent.
  We show the implications (a)$+$(c) $\Rightarrow$ (b) and (b)
  $\Rightarrow$ (a).

  (a)$+$(c) $\Rightarrow$ (b): Since $A_\infty$ is
  $\Sym$-equivariantly finitely generated (by (c)), it follows from
  Corollary~\ref{c:fg-local-global} that $\A$ is eventually saturated.
  Thus, there exists $r\in \NN$ such that
  $A_\infty \cap R_{n} = A_{n}$ for all $n\ge r$, which is (b)(i).  By
  (a), $\ind(\A)$ is finite and $A_n$ is finitely generated for
  $n\gg0$.  So we may assume that $r\ge\ind(\A)$ and $A_r$ is finitely
  generated by, say, $v_1,\dots,v_s\in A_r$.  Then for all $n\geq r$
  we have
  $A_n=(\Sym_{r,n}(v_1)\cup\dots \cup\Sym_{r,n}(v_s))^{\cl_n}$.  Each
  $v_t \in A_r\subseteq R_r$ has support size at most~$r$.  Since
  $\Sym_{r,n}$ is finite and its elements do not change the support
  size of $v_t$ by Lemma~\ref{l:Sym-Sym}(i), we conclude~(b)(ii).

  (b) $\Rightarrow$ (a): Assuming $r$ as in (b), it suffices to show
  that for all $n\geq r$, $ (\Sym_{r,n}(A_r))^{\cl_n} = A_n$.  This
  is trivially true for $n=r$. Let $n>r$.  By (b)(ii), $A_n$ is
  generated by some $u_1,\dots,u_p \in R_n$ of support size at
  most~$r$.  We want to show that
  $u_{1},\dots, u_{p} \in (\Sym_{r,n}(A_r))^{\cl_n}$.  For any
  $t\in[p]$, $|\supp(u_t)|\leq r<n$ and thus there exists a
  $\sigma \in \Sym(n)$ such that $\sigma(u_t)\in R_r$.  Hence,
  $\sigma(u_t)\in R_{r} \cap A_\infty = A_{r}$, using (b)(i).  Then by
  Lemma~\ref{l:Sym-Sym}(ii),
  \[
    u_t =\sigma^{-1}(\sigma(u_t))\in \Sym(n) (A_{r})=\Sym_{r,n}
    (A_{r}) \subseteq (\Sym_{r,n} (A_{r}))^{\cl_n}.\qedhere
  \]
\end{proof}


For $\Inc$-invariant chains we obtain a weaker version of
Theorem~\ref{thm:setFG}. The polynomial ring case of this result and
its relation to \cite[Theorem~3.6]{hillar2012finite} are discussed in
Section~\ref{sec:ideals}.

\begin{theorem}
\label{t:local-global-Inc}
Let $(\ccl,\cl_\infty)$ be a consistent system of closure operations
so that $\ccl$ is compatible with~$\Inc$.  Let $\A=(A_n)_{n\geq 1}$ be
an $\Inc$-invariant, $\ccl$-closed chain such that the limit set
$A_\infty$ is $\cl_{\infty}$-closed.  Consider the following
statements:
\begin{enumerate}[(a)]
\item
  (local)
  $\A$ stabilizes and is eventually finitely generated.
\item
  (global) $A_\infty$ is $\Inc$-equivariantly finitely generated.
\end{enumerate}
Then (a) implies (b) and if $\A$ is eventually saturated, then (a) and
(b) are equivalent.
\end{theorem}

\begin{proof}
  That (a) implies (b) follows immediately from
  Theorem~\ref{th:fg-and-stable}(i) and
  Lemma~\ref{l:Sym-Inc-consistency}(i).

  To prove the equivalence, we can assume that $\A$ is saturated.  To
  see this, replace $\A$ by its saturation~$\ol\A$.  Then (b) and the
  assumptions of the theorem are still satisfied by
  Lemma~\ref{l:invariant}(ii); (a) remains unchanged since $\A $ and
  $\ol\A$ coincide eventually (because $\A$ is eventually
  saturated). Therefore, it is harmless to assume that $\A$ is
  saturated.

  We prove (b) $\Rightarrow$ (a) using
  Proposition~\ref{pr:relaxation}.  To verify its assumptions, use
  Lemma~\ref{l:Sym-Inc-locally-finite} for the local finiteness of
  $\Inc$.  Then by (b), there exist $w_1,\dots,w_s\in A_\infty$ such
  that
  \[
    A_\infty=(\Inc(w_1)\cup\dots\cup \Inc(w_s))^{\cl_{\infty}}.
  \]
  Choose $m$ so that $w_t \in A_m$ for all $t\in [s]$.  To apply
  Proposition~\ref{pr:relaxation}, it remains to show that
  \begin{equation}
    \label{eq:Inc-inclusion}
    \Inc(w_t)\cap R_n\subseteq \Inc_{m,n}(A_{m})\
    \text{ for all $n\ge m$ and $t\in [s]$.}
  \end{equation}
  Let $k=k_t\defas\width(w_t)$.  It is evident that $k\le m$ and
  $w_t \in A_\infty\cap R_{k}=A_{k}$, where we used that $\A$ is
  saturated.  By Lemma~\ref{l:truncated-orbit}(i),
  $\Inc(w_t)\cap R_n = \Inc_{k,n}(w_t) \subseteq \Inc_{k,n}(A_{k})$.
  To prove \eqref{eq:Inc-inclusion} it suffices to show that
  \[
    \Inc_{k,n}(A_{k}) \subseteq \Inc_{m,n}(A_m) \ \text{ for all } k\le m.
  \]
  By Lemma~\ref{l:Inc}, $\Inc_{k,n}=\Inc_{k+1,n}\comp \Inc_{k,k+1}$, and
  hence
  \[
    \Inc_{k,n}(A_{k}) = (\Inc_{k+1,n}\comp \Inc_{k,k+1}) (A_{k})
    \subseteq \Inc_{k+1,n}(A_{k+1}).
  \]
  The proof finishes with a finite induction.
\end{proof}

\begin{remark}
\label{r:extension}
Comparing Theorems~\ref{thm:setFG} and~\ref{t:local-global-Inc},
the following questions are natural:
\begin{enumerate}[(1)]
\item Can the assumption that $\A$ is eventually saturated in
  Theorem~\ref{t:local-global-Inc} be omitted?
\item Does there exist a characterization for $\Inc$-invariant chains
  similar to Theorem~\ref{thm:setFG}(b)? Precisely, when $\A$ is
  eventually saturated, is it true that $A_\infty$ is
  $\Inc$-equivariantly finitely generated if $\A$ is eventually
  finitely generated by elements of bounded support size?
\end{enumerate}
These questions can be affirmatively answered for chains of ideals
(Theorem~\ref{thm:Inc-ideal-stable}), but in general both have
negative answers as we show in Examples~\ref{ex:Inc-not-stable}
and~\ref{ex:Inc-support-size}.
\end{remark}

\section{Cones and monoids up to symmetry}
\label{section-cones-monoids}

In this section we specialize the results of the previous section to
invariant chains of convex cones and monoids.  We provide various
examples to demonstrate that some assumptions of our results are
indispensable and give counterexamples to potential strengthenings.

The results of
this section have recently been employed to extend foundational
results in polyhedral geometry to the equivariant setting~\cite{LR}.
For standard terminology on cones and monoids the reader is referred
to \cite{bruns2009book} and~\cite{ziegler1995book}.  
\subsection{Cones up to symmetry}
We consider convex cones in the nonnegative orthant
$\RR^{(\NN)}_{\geq 0}$ of~$\RR^{(\NN)}$.  That is, the ambient chain
is $\R=(R_{n})_{n\geq 1}$ with $R_n=\RR^{n}_{\geq 0}$ for all $n\ge 1$
and $R_\infty = \bigcup_{n\ge1}\RR^{n}_{\geq 0}=\RR^{(\NN)}_{\geq 0}$.
To consider convex cones, the closure operations are
$\cncl_n=\cone(\cdot)$ in $\RR^{n}_{\geq 0}$ for $n\geq 1$ and
$\cncl_\infty=\cone(\cdot)$ in $R_{\infty}$, where $\cone(A)$ consists
of finite nonnegative linear combinations from~$A$, that is,
$\cone(A)=\{\sum_{i=1}^k\lambda_ia_i\mid k\in\NN,\ a_i\in A,\
\lambda_i\in\RR_{\geq 0}\}$.
We restrict our attention to cones in $R_{\infty}$ instead of the
whole space $\RR^{(\NN)}$, since the system of conical hulls
$(\ccn,\cncl_\infty)$ is consistent on $R_{\infty}$, but not
on~$\RR^{(\NN)}$.

\begin{lemma}
\label{l:cone-consistency}
Consider the ambient chain $\R=(\RR^{n}_{\geq 0})_{n\ge1}$ of
nonnegative orthants as above. Then the system of conical hulls
$(\ccn,\cncl_\infty)$ is consistent.
\end{lemma}

\begin{proof}
  For any $A\subseteq \RR^{(\NN)}_{\geq 0}$ and $n\in\NN$ we need to
  show that
  $\cncl_n(A\cap\RR^{n}_{\geq 0})=\cncl_\infty(A)\cap\RR^{n}_{\geq
    0}$.  The inclusion ``$\subseteq$'' is obviously true. For the
  reverse inclusion, take $v\in\cncl_\infty(A)\cap\RR^{n}_{\geq
    0}$. Then $\width(v)\le n$ and there exist $a_1,\dots,a_k\in A$
  and $\lambda_1,\dots,\lambda_k>0$ such that
  $ v=\sum_{i=1}^k\lambda_ia_i$.  By positivity of the $\lambda_{i}$,
  and since $a_i\in \RR^{(\NN)}_{\geq 0}$ for all $i\in[k]$, we have
  $\width(a_i)\le\width(v)\le n$.  Thus $a_i\in A\cap\RR^{n}_{\geq 0}$
  for all $i\in[k]$, and hence
  $v=\sum_{i=1}^k\lambda_ia_i\in \cncl_n(A\cap\RR^{n}_{\geq 0}).$
\end{proof}

It is easy to give examples showing that $(\ccn,\cncl_\infty)$ is not
consistent on $\RR^{(\NN)}$.

\begin{example}
  \label{ex:cone-inconsistency}
  Let $A=\{(1,1),(1,-1)\}\subset\RR^2$.  Identifying $\RR$ with
  $\RR\times\{0\}\subseteq\RR^2$, one has
  \[
    \cncl_1(A\cap\RR)=\cncl_1(\emptyset)=\{0\}\subsetneq
    \RR_{\ge0}=\cncl_2(A)\cap\RR.
  \]
  By Remark~\ref{r:local-consistency}, this means that the system
  $(\ccn,\cncl_\infty)$ is not consistent if we choose the ambient
  chain with $R_n=\RR^n$ for all $n\ge1$.
\end{example}

The next result shows that the chain $\ccn$ of conical hulls is
compatible with $\Sym(\infty)$ and $\Inc$, even when one considers the
whole space $\RR^{(\NN)}$.

\begin{lemma}
\label{l:cone-compatibility}
Let $\Pi=\Sym(\infty)$ or $\Pi=\Inc$. Then for any $m\le n$ and
$A_m\subseteq\RR^m$ one has
$\Pi_{m,n}(\cncl_m(A_m))\subseteq \cncl_n(\Pi_{m,n}(A_m))$.  Thus,
with $\R=(\RR^{n}_{\geq 0})_{n\ge1}$, $\ccn$ is compatible with~$\Pi$.
\end{lemma}

\begin{proof}
  Let $\pi\in\Pi_{m,n}$ and
  $v=\sum_{i=1}^k\lambda_ia_i\in \cncl_m(A_m)$
  with $a_i\in A_m$ and $\lambda_i\ge0$.
  We expand $a_i=\sum_{j=1}^ma_{ij}\epsilon_{j}$ in the canonical
  basis.  Then by~\eqref{eq:pi},
  \[
    \pi(v)
    =
    \pi\Big(\sum_{j=1}^m\big(\sum_{i=1}^k\lambda_ia_{ij}\big)\epsilon_{j}\Big)
    =
    \sum_{j=1}^m\big(\sum_{i=1}^k\lambda_ia_{ij}\big)\epsilon_{\pi(j)}
    =
    \sum_{i=1}^k\lambda_i\pi(a_i).
  \]
  This yields $\pi(v)\in\cncl_n(\Pi_{m,n}(A_m))$, as desired.
\end{proof}

It is obvious that
for an increasing chain $\C=(C_{n})_{n\geq 1}$ of convex cones with
$C_{n} \subseteq \RR^{n}$ for all $n\ge1$, the limit set $C_\infty$ is
a convex cone.
With Lemmas~\ref{l:cone-consistency} and \ref{l:cone-compatibility} in
place, we are ready to apply the results of the previous section to
chains of cones.  At first we formulate the following version of
Theorem~\ref{thm:setFG}.

\begin{corollary}
  \label{c:globallocalcones}
  Let $\C=(C_{n})_{n\geq 1}$ be a $\Sym$-invariant chain of convex cones
  $C_{n} \subseteq \RR^{n}_{\geq 0}$ with limit cone~$C_\infty$.  Then
  the following statements are equivalent:
  \begin{enumerate}[(a)]
  \item
    $\C$ stabilizes and is eventually finitely generated;
  \item
    There exists an $r\in\NN$ such that for all $n\ge r$ the following hold:
    \begin{enumerate}[(i)]
    \item
      $C_\infty \cap \RR^{n}_{\geq 0} = C_{n}$,
    \item
      $C_{n}$ is finitely generated by elements of support size at most $r$;
    \end{enumerate}
  \item
    $C_\infty$ is $\Sym$-equivariantly finitely generated.
  \end{enumerate}
\end{corollary}

We now give several examples that demonstrate the necessity of the
conditions in Corollary~\ref{c:globallocalcones}.  The first shows
that in (a) and (b) the word ``eventually'' cannot be omitted.

\begin{example}
\label{e:limitconevschain-beginning}
Consider the chain $\C=(C_n)_{n\geq 1}$ with $C_{1} = \{0\},$
\[
C_{2} =\{ (x,y)\in \RR^2_{\geq 0}\mid x>0, \ y>0\} \cup \{(0,0)\}
\ \text{ and }\
C_{n} = \RR^n_{\geq 0}\ \text{ for } n\ge 3.
\]
Then $\C$ stabilizes as a $\Sym$-invariant chain of convex cones.  The
limit cone
$C_\infty = \RR^{(\NN)}_{\geq 0} = \cncl_\infty(\Sym (\epsilon_1))$ is
$\Sym$-equivariantly generated by the first basis vector.  However,
this does not imply that all cones $C_n$ are finitely generated, since
$C_2$ is not.  Additionally
$C_\infty\cap \RR^2_{\geq 0}=\RR^2_{\geq 0} \neq C_2$. Thus, $\C$ is
not saturated.
\end{example}

The only insight in this example is that finite subsequences can
often be changed rather wildly without changing the limit~$C_\infty$.
This suggests that properties of $C_\infty$ can only be related to the
tail of the chain, or alternatively, to the unique saturated chain
defining it.

The following example shows that for the implication ``(a)
$\Rightarrow$ (c)'' in Corollary~\ref{c:globallocalcones} it is
necessary to assume that the cones are eventually finitely generated.

\begin{example}
\label{e:counterex-fg}
Consider the chain $\C=(C_n)_{n\geq 1}$ with $C_{1} = \{0\}$,
\[
C_{2} =\{ (x,y)\in \RR^2_{\geq 0}\mid x>0, \ y>0\} \cup \{(0,0)\}
\ \text{ and  }\
C_{n} = \cncl_n({\Sym}_{2,n}(C_2))
 \ \text{ for } n\ge 3.
\]
By construction, $\C$ stabilizes as a $\Sym$-invariant chain of convex
cones.  However, $C_\infty$ is not $\Sym$-equivariantly finitely
generated, since then $C_\infty\cap \RR^2_{\geq 0}$ would be finitely
generated.  But $C_\infty\cap \RR^2_{\geq 0}=C_2$ because $C_2$
contains exactly the elements of $C_\infty$ of width at most two, and
$C_2$ is evidently not finitely generated.
\end{example}

We now show that in Corollary~\ref{c:globallocalcones}(b) both (i) and
(ii) are necessary and independent. 

\begin{example}
\label{e:no-global-fg-Cone1}
For $n\ge 2$ let $v_n=(n,1,0,\dotso,0)\in \RR^{n}_{\geq 0}$. Consider
the chain $\C=(C_n)_{n\geq 1}$ with $C_{1} = \{0\}$,
$C_2 = \cncl_2(\Sym(2)(v_2))$ and
\[
  C_n =
  \cncl_n\big(\Sym(n)\big(\iota_{n-1,n}(C_{n-1})\cup\{v_n\}\big)\big)
  \subseteq \RR^{n}_{\geq 0} \ \text{ for }\ n\ge 3.
\]
Then $\C$ is a $\Sym$-invariant chain of convex cones and the
following hold:
\begin{enumerate}
\item All cones $C_n$ are pointed, rational and finitely generated (as
  a cone) by one element of support size two up to symmetry.  More
  precisely, for $n\geq 2$ we have $C_n = \cncl_n(\Sym(n) (v_n))$.
  In particular, the orbits of all generators of $C_m$ in $C_{n}$ are
  redundant whenever $m<n$.
\item $C_\infty$ is not $\Sym$-equivariantly finitely generated.
\item $C_\infty \cap \RR^n_{\geq 0} \neq C_n$ for all $n\geq 2$.
\end{enumerate}

\begin{proof}
  (i) By definition,
  $\cncl_{n}(\Sym_{n-1,n}(C_{n-1})) \subseteq C_{n}$ for $n\geq 2$.
  Thus, $\C$ is indeed a $\Sym$-invariant chain of pointed and
  rational convex cones. Let
  $ \widetilde{C}_n=\cncl_n(\Sym(n) (v_n)) $ for $n\ge2$.  Clearly,
  $\widetilde{C}_n \subseteq C_n$ for every $n\geq 2$.  We show by
  induction that equality holds, which is trivially true for $n = 2$.
  Let $n > 3$ and assume that ${C}_{n-1} = \widetilde C_{n-1}$.  It
  suffices to prove that
\begin{equation}
 \label{eq:cone-inclusion}
 \cncl_n(\Sym_{n-1,n}(C_{n-1}))
\subseteq \widetilde{C}_n.
\end{equation}
Using Lemma~\ref{l:cone-compatibility} and the fact that
$\Sym_{n-1,n} = \Sym_{n-1,n}\comp \Sym(n-1)$ one has
\begin{align*}
\cncl_n(\Sym_{n-1,n}(C_{n-1}))&=
  \cncl_n(\Sym_{n-1,n}(\cncl_{n-1}(\Sym(n-1) (v_{n-1}))))\\
  &\subseteq \cncl_n(\Sym_{n-1,n}(\Sym(n-1) (v_{n-1})))\\
  &= \cncl_n(\Sym_{n-1,n}(v_{n-1}))
  = \cncl_n(\Sym(n)(\iota_{n-1,n}(v_{n-1}))).
\end{align*}
Since $v_{n-1}\in C_{n-1}$, this implies
\begin{equation}
 \label{eq:cone-equality}
 \cncl_n(\Sym_{n-1,n}(C_{n-1}))= \cncl_n(\Sym(n)(\iota_{n-1,n}(v_{n-1}))).
\end{equation}
So for \eqref{eq:cone-inclusion} we only need to show that
$\iota_{n-1,n}(v_{n-1})\in \widetilde{C}_n$.  Applying $\iota_{2,n}$
to
\[
  (n-1,1)=\frac{n^2-n-1}{n^2-1} (n,1)+\frac{1}{n^2-1} (1,n)
\]
we obtain
\[
  \iota_{n-1,n}(v_{n-1})=\frac{n^2-n-1}{n^2-1} v_n+\frac{1}{n^2-1} \sigma(v_n),
\]
where $\sigma\in\Sym(n)$ is the transposition~$(1\,2)$. Hence,
$\iota_{n-1,n}(v_{n-1})\in \cncl_n(\Sym(n) (v_n))=\widetilde{C}_n$,
as desired.  So
${C}_{n} = \widetilde C_{n}$ for $n\geq 2$ and the orbits of all
generators of $C_{n-1}$ in $C_{n}$ are redundant.  It follows that the same is
true for the cones $C_m$ and~$C_n$, whenever $m<n$.

(ii) We claim that $\C$ does not stabilize. Indeed, it suffices to
show that
\begin{equation*}
  v_n \in C_n \setminus \cncl_n(\Sym_{n-1,n}(C_{n-1}))
\ \text{ for all }\ n\geq 2.
\end{equation*}
Assume the contrary.  Then $v_n \in \cncl_n(\Sym_{n-1,n}(C_{n-1}))$
for some $n\ge 2$. So by \eqref{eq:cone-equality}, there exist
$\lambda_1,\dots,\lambda_k>0$ and $\sigma_1,\dots,\sigma_k\in\Sym(n)$
such that
\begin{equation}
 \label{eq:cone-stable}
 v_n=\sum_{i=1}^k\lambda_i\sigma_i(\iota_{n-1,n}(v_{n-1})).
\end{equation}
By the positivity of the $\lambda_{i}$,
$\width(\sigma_i(\iota_{n-1,n}(v_{n-1})))\le 2$ for all $i\in[k]$.
It follows that
\[
 \sigma_i(\iota_{n-1,n}(v_{n-1}))=(n-1,1,0,\dots,0)
 \ \text{ or }\
 \sigma_i(\iota_{n-1,n}(v_{n-1}))=(1,n-1,0,\dots,0)
\]
for all $i\in[k].$ This together with \eqref{eq:cone-stable} yields
the existence of $\lambda,\mu\ge0$ such that
\[
 (n,1)=\lambda(n-1,1)+\mu(1,n-1).
\]
But this is impossible. Hence, the chain $\C$ does not stabilize.

By Corollary~\ref{c:globallocalcones}, $C_\infty$ is not
$\Sym$-equivariantly finitely generated.  One can also show that
$C_\infty$ is the subset of $\RR^{(\NN)}_{\geq 0}$ consisting of $0$
and all elements of support size at least two, leading to the
following explicit description:
\[
  C_\infty =
  \RR^{(\NN)}_{\geq 0}
  \setminus
  \Big(\bigcup_{n\geq 1}
  \RR_{>0}\,\epsilon_n\Big).
\]

(iii) Since $\width(v_{n+1})=2$, we have
$ v_{n+1}\in C_\infty \cap \RR^{2}_{\geq 0} \subseteq C_\infty \cap
\RR^{n}_{\geq 0}$ for $n\geq 2$.  On the other hand, it has been shown
in (ii) that $v_{n+1}\not\in C_n$.  Hence,
$ C_\infty \cap \RR^{n}_{\geq 0} \neq C_n \ \text{ for all }\ n\geq
2$.
\end{proof}
\end{example}

The next example gives a chain with no global bound for the support
sizes of generators.

\begin{example}
\label{e:no-global-fg-Cone2}
For $n\ge 2$, let $w_n=(n-1,1,\dotso,1)\in \RR^{n}_{\geq 0}$.
Consider the chain $\C=(C_n)_{n\geq 1}$ with $C_{1} = \{0\}$,
$C_2 = \cncl_2(\Sym(2)(w_2))$ and
\[
  C_n =
  \cncl_n\big(\Sym(n)\big(\iota_{n-1,n}(C_{n-1})\cup\{w_n\}\big)\big)
  \subseteq \RR^{n}_{\geq 0} \ \text{ for }\ n\ge 3.
\]
Then $\C$ is a $\Sym$-invariant chain of convex cones and the
following hold:
\begin{enumerate}
\item All cones $C_n$ are pointed, rational and finitely generated (as
  cones), but there exists no global bound for the support sizes of
  generators of the cones.
\item $C_\infty$ is not $\Sym$-equivariantly finitely generated.
\item $C_\infty \cap \RR^n_{\geq 0} = C_n$ for all $n\geq 2$.
\end{enumerate}

\begin{proof}
  (i) As for Example~\ref{e:no-global-fg-Cone1}(i) we see that $\C$ is
  a $\Sym$-invariant chain of pointed, rational and finitely generated
  convex cones. An easy induction shows that $C_n$ is generated by
  \[
    G_n=\{\Sym_{m,n}(w_m)\mid 2\le m\le n\}.
  \]
  In this generating set, $w_n$ is irredundant for $C_n$ whenever
  $n\ge3$.  To see this, consider the supporting hyperplane
  $H=\{(z_1,\dots,z_n)\in\RR^n\mid z_1=\sum_{l=2}^n z_l\}$ of~$C_n$.
  If $w_n$ were redundant, there would exist
  $\lambda_1,\dots,\lambda_k>0$ and
  $u_1,\dots,u_k\in G_n\setminus\{w_n\}$ such that
  $w_n=\sum_{i=1}^k\lambda_iu_i$.  Since $w_n$ lies on the supporting
  hyperplane~$H$, it follows that $u_i\in H$ for all $i=1,\dots,n$.
  Thus, each $u_i$ must have $m-1$ as its first entry and exactly
  $(m-1)$ entries $1$ as the remaining nonzero entries for some $m<n$.
  Comparing the second entry in $w_n=\sum_{i=1}^k\lambda_iu_i$, we
  find $\sum_{i=1}^k\lambda_i\leq 1$.  Now comparing the first entry
  yields
  $n-1=\sum_{i=1}^k\lambda_iu_{i,1} < \sum_{i=1}^k\lambda_i(n-1)\leq
  n-1$, which is impossible.  Hence, $C_n$ has an irredundant generator of support
  size~$n$.

  (ii) By (i) and Corollary~\ref{c:globallocalcones}, $C_\infty$ is
  not $\Sym$-equivariantly finitely generated.

  (iii) According to Remark~\ref{r:local-consistency}, the consistency
  of $(\ccn,\cncl_\infty)$ (by Lemma~\ref{l:cone-consistency}) gives
	\[
	C_k\cap\RR^{n}_{\geq 0}
	=
	\cncl_k(G_k)\cap\RR^{n}_{\geq 0}
	=
	\cncl_n\big(G_k\cap\RR^{n}_{\geq 0}\big)
	=
	\cncl_n(G_n)
	=
	C_n
	\]
	for all $k\ge n\ge 2.$ It follows that
	\[
	C_\infty\cap\RR^{n}_{\geq 0}
	=
	\bigcup_{k\geq n}C_k\cap\RR^{n}_{\geq 0}
	=
	\bigcup_{k\geq n}\big(C_k\cap\RR^{n}_{\geq 0}\big)
	=
	 C_n
	 \ \text{ for all }\  n\ge 2.\qedhere
	\]
\end{proof}
\end{example}

For the reader's convenience, we briefly summarize the above examples
in Table~\ref{t:summary}.

\begin{table}[bhtp]
\begin{tabular}{|c|c|c|c|c|c|}
\hline
\multirow{2}{*}{\vspace{-1cm}Example}
&\multicolumn{2}{|c|}{(a)}                                                       & \multicolumn{2}{c|}{(b)}                                                                                                                                               & (c)
\\
\cline{2-6}
&$\C$ stabilizes
& \begin{tabular}[c]{@{}c@{}}$\C$  evtl.\ f.\ g.\end{tabular}
& \begin{tabular}[c]{@{}c@{}}$\C$  evtl.\ satur.\end{tabular}
& \begin{tabular}[c]{@{}c@{}}$\C$  evtl.\ f.\ g.\ by\\ elem.\ of  bound.\\ supp.\ size\end{tabular}
& \begin{tabular}[c]{@{}c@{}}$C_\infty$ $\Sym$-\\ equiv.\ f.\ g.\end{tabular}
\\
\hline
\ref{e:limitconevschain-beginning}
&\cmark
& \begin{tabular}[c]{@{}c@{}}\cmark  \\ ($C_2$ not f.\ g.)\end{tabular}
& \begin{tabular}[c]{@{}c@{}}\cmark  \\ ($\C$ not satur.)\end{tabular}
& \cmark                                                                                                   & \cmark
\\
\hline
 \ref{e:counterex-fg}
&\cmark
&\xmark
& \cmark                                                                 & \xmark                                                                                                  & \xmark
\\
\hline
\ref{e:no-global-fg-Cone1}
&\xmark
& \cmark                                                                 & \xmark                                                                 & \cmark                                                                                                   & \xmark
\\
\hline
\ref{e:no-global-fg-Cone2}
&\xmark
& \cmark                                                                 & \cmark                                                                  & \xmark                                                                                                  & \xmark
\\
\hline
\end{tabular}
 \medskip
\caption{Summary of Examples \ref{e:limitconevschain-beginning}--\ref{e:no-global-fg-Cone2}\label{t:summary}}
\end{table}

We consider $\Inc$-invariant chains next. By
Lemmas~\ref{l:cone-consistency} and \ref{l:cone-compatibility}, we
obtain the following cone version of Theorem~\ref{t:local-global-Inc}.

\begin{corollary}
\label{c:cone-local-global-Inc}
Let $\C=(C_{n})_{n\geq 1}$ be an $\Inc$-invariant chain of convex
cones $C_{n} \subseteq \RR^{n}_{\geq 0}$ with limit $C_\infty$.
Consider the following statements:
\begin{enumerate}[(a)]
\item $\C$ stabilizes and is eventually finitely generated.
\item $C_\infty$ is $\Inc$-equivariantly finitely generated.
\end{enumerate}
Then (a) implies (b) and if $\C$ is eventually saturated, then (a) and
(b) are equivalent.
\end{corollary}

We again give counterexamples to potential strengthenings of the
corollary.  The first shows that, in contrast to
Corollary~\ref{c:globallocalcones}(b)(i), the equivariant finite
generation of the limit cone $C_\infty$ does not imply that $\C$ is
eventually saturated, and moreover, that the equivalence in
Corollary~\ref{c:cone-local-global-Inc} does not hold without the
assumption that $\C$ is eventually saturated.  By
Theorem~\ref{th:fg-and-stable}(ii), this also implies that the system
$(\cInc,\Inc)$ is not consistent.

\begin{example}
  \label{ex:Inc-not-stable}
  For $n\ge2$ let
  \[ A_n=\{\epsilon_1,\dots,\epsilon_{n-1}\}
    \ \text{ and  } \
    B_n=\{\epsilon_i+n\epsilon_{n}\mid 1\le i\le n-1\}.
  \]
  Consider the chain $\C=(C_n)_{n\geq 1}$ with $C_{1} = \{0\}$ and
  $ C_{n} = \cncl_n(A_n\cup B_n)\subseteq \RR^{n}_{\geq 0}$ for
  $n\ge 2$.  Then the following hold:
\begin{enumerate}
\item $\C$ is an $\Inc$-invariant chain of convex cones and each $C_n$
  is finitely generated by vectors of support size at most two.
\item $C_\infty$ is $\Inc$-equivariantly finitely generated.
\item $\C$ does not stabilize.
\item $C_\infty \cap \RR^n_{\geq 0} \neq C_n$ for all $n\geq 2$.
\end{enumerate}

\begin{proof}
  (i) For $n\ge2$ let
  $ \widetilde{B}_{n+1}=\{\epsilon_i+n\epsilon_{n+1}\mid 1\le i\le
  n\}$.  We claim that
  \[
    \cncl_{n+1}(\Inc_{n,n+1}(C_n))
    =\cncl_{n+1}(A_{n+1}\cup \widetilde{B}_{n+1})
    \ \text{ for }\ n\ge 2.
  \]
  One has $\Inc_{n,n+1}(A_n) = A_{n+1}$ and
  $\Inc_{n,n+1}(B_n) = B_{n}\cup \widetilde{B}_{n+1}$.  Evidently,
  $ B_n\subseteq\cncl_{n+1}(A_{n+1})$ and therefore,
  \begin{align*}
    \cncl_{n+1}({\Inc}_{n,n+1}(C_n))
    &=
      \cncl_{n+1}({\Inc}_{n,n+1}(\cncl_{n}(A_{n}\cup B_n)))
      =
      \cncl_{n+1}({\Inc}_{n,n+1}(A_{n}\cup B_n))\\
    &=
      \cncl_{n+1}(A_{n+1}\cup B_{n}\cup \widetilde{B}_{n+1})
      =
      \cncl_{n+1}(A_{n+1}\cup \widetilde{B}_{n+1}),
\end{align*}
where we used Lemma~\ref{l:compatibility} in the second equality and
also that conical hulls are compatible with $\Inc$ by
Lemma~\ref{l:cone-compatibility}.  Now since
$
(n+1)(\epsilon_i+n\epsilon_{n+1})=\epsilon_i+n(\epsilon_i+(n+1)\epsilon_{n+1})$,
it follows that
$\widetilde{B}_{n+1}\subseteq\cncl_{n+1}(A_{n+1}\cup
{B}_{n+1})$. Hence
\[
  \cncl_{n+1}(\Inc_{n,n+1}(C_n))
  \subseteq\cncl_{n+1}(A_{n+1}\cup {B}_{n+1})
  =C_{n+1}
  \ \text{ for }\ n\ge 2,
\]
i.e.\ $\C$ is an $\Inc$-invariant chain of convex cones. The remaining
assertion is obvious.

(ii)
Since
\[
  \RR^{(\NN)}_{\geq 0}
  =
  \bigcup_{n\ge2}\cncl_n(A_n)
  \subseteq
  \bigcup_{n\ge2} C_n
  =
  C_\infty
  \subseteq
  \RR^{(\NN)}_{\geq 0},
\]
one obtains
$C_\infty = \RR^{(\NN)}_{\geq 0} = \cncl_\infty(\Inc(\epsilon_1))$.
Thus, $C_\infty$ is $\Inc$-equivariantly generated by~$\epsilon_1$.

(iii) For all $n\ge 2$ and $1\le i\le n$ one has
\[
  \epsilon_i+(n+1)\epsilon_{n+1}
  \not\in\cncl_{n+1}(A_{n+1}\cup \widetilde{B}_{n+1})
  =\cncl_{n+1}(\Inc_{n,n+1}(C_n)).
\]
This implies $\cncl_{n+1}(\Inc_{n,n+1}(C_n))\varsubsetneq C_{n+1}$ for
$n\ge 2$, which means that $\C$ does not stabilize.

(iv) The assertion follows since
$\epsilon_n\in\RR^{n}_{\geq 0} = C_\infty\cap\RR^{n}_{\geq 0}$, but
$\epsilon_n\not\in C_n$ for $n\ge2$.
\end{proof}
\end{example}

As announced in Remark~\ref{r:extension}, the equivalence ``(b)
$\Leftrightarrow$ (c)'' in Corollary~\ref{c:globallocalcones} cannot
be extended to $\Inc$-invariant chains.  The following example shows
this and also that the equivalence ``(a) $\Leftrightarrow$ (b)'' in
Corollary~\ref{c:globallocalcones} does not hold for $\Inc$-invariant
chains.

\begin{example}
  \label{ex:Inc-support-size}
  For $n\ge2$ let
  \[
    E_n=\{\epsilon_2,\dots,\epsilon_{n}\} \ \text{ and } \
    F_n=\{i\epsilon_1+\epsilon_{i}\mid 2\le i\le n\}.
  \]
  Consider the chain $\C=(C_n)_{n\geq 1}$ with $C_{1} = \{0\}$ and
  $C_{n} = \cncl_n(E_n\cup F_n)\subseteq \RR^{n}_{\geq 0}$ for
  $n\ge 2$.  Then the following hold:
  \begin{enumerate}
  \item $\C$ is $\Inc$-invariant and each $C_n$ is finitely generated
    in support size at most two.
  \item $C_\infty \cap \RR^n_{\geq 0} = C_n$ for all $n\geq 2$.
  \item $\C$ does not stabilize.
  \item $C_\infty$ is not $\Inc$-equivariantly finitely generated.
\end{enumerate}

\begin{proof}
  (i) For $n\ge2$ let
  $ \widetilde{F}_{n+1}=\{i\epsilon_1+\epsilon_{i+1},\
  i\epsilon_2+\epsilon_{i+1}\mid 2\le i\le n\}$.
  It is evident that
  \begin{align*}
    \Inc_{n,n+1}(E_n)
    = E_{n+1}
    \ \text{ and  } \
    \Inc_{n,n+1}(F_n)
    = F_{n}\cup \widetilde{F}_{n+1}.
  \end{align*}
  Similarly to Example~\ref{ex:Inc-not-stable} one can show that
  \[
    \cncl_{n+1}(\Inc_{n,n+1}(C_n))
    =\cncl_{n+1}(E_{n+1}\cup F_{n}\cup \widetilde{F}_{n+1})
    \ \text{ for }\ n\ge 2.
  \]
  Since $F_{n}\subseteq F_{n+1}$ and
  $\widetilde{F}_{n+1}\subseteq\cncl_{n+1}(E_{n+1}\cup {F}_{n+1})$, it
  follows that
  \begin{align*}
    \cncl_{n+1}({\Inc}_{n,n+1}(C_n))
    \subseteq
    \cncl_{n+1}(E_{n+1}\cup {F}_{n+1})=C_{n+1}.
  \end{align*}
  Hence, $\C$ is an $\Inc$-invariant chain of convex cones. The
  remaining assertion is obvious.
	
  (ii) Similar to Example~\ref{e:no-global-fg-Cone2}(iii) one has
  \[
    C_k\cap\RR^{n}_{\geq 0}
    =
    \cncl_k(E_k\cup F_k)\cap\RR^{n}_{\geq 0}
    =
    \cncl_n\big((E_k\cup F_k)\cap\RR^{n}_{\geq 0}\big)
    =
    \cncl_n(E_n\cup F_n)
    =
    C_n
  \]
  for all $k\ge n\ge 2$.  Therefore,
  \[
    C_\infty\cap\RR^{n}_{\geq 0}
    =
    \bigcup_{k\geq n}C_k\cap\RR^{n}_{\geq 0}
    =
    \bigcup_{k\geq n}\big(C_k\cap\RR^{n}_{\geq 0}\big)
    =
    C_n
    \ \text{ for all }\  n\ge 2.
  \]
	
  (iii) $\C$ does not stabilize since
  $(n+1)\epsilon_1+\epsilon_{n+1} \in C_{n+1}\setminus
  \cncl_{n+1}(\Inc_{n,n+1}(C_n))$ for all $n\ge 2$.
	
  (iv) $C_\infty$ is not $\Inc$-equivariantly finitely generated by
  Corollary~\ref{c:cone-local-global-Inc}.
\end{proof}
\end{example}


All cones considered here are given in their
\emph{$V$-representation}, i.e., generated by vectors.  This is
suitable to consider finite generation of cones.  By the
Minkowski--Weyl theorem \cite[Theorem~1.3]{ziegler1995book}, finitely
generated cones have a finite $H$-representation as an intersection of
linear halfspaces.  It would be interesting to understand the
interaction of this duality with $\Pi$-equivariance.  To work on the
following problem, it might be necessary to replace the ambient space,
which is a direct limit of $(\RR^{n}_{\ge 0})_{n\ge 1}$, with an inverse limit.
In general, inverse limits as ambient spaces provide many interesting
directions for the future.

\begin{problem}
  Let a chain of convex cones be given in their
  $H$-representations.  Find $H$-versions of
  Corollaries~\ref{c:globallocalcones}
  and~\ref{c:cone-local-global-Inc} and relate them to the
  $V$-versions provided there.
  \footnote{This problem has been resolved for $\Sym$-invariant chains of cones in~\cite{LR}.}
\end{problem}

\subsection{Monoids up to symmetry}

In almost complete analogy to the cones discussed above, one can also
consider monoids in $\ZZ_{\geq 0}^{(\NN)}$.  Then the ambient chain is
$\R=(\ZZ_{\geq 0}^n)_{n\ge1}$ with ambient set
$R_\infty=\bigcup_{n\ge1}\ZZ^{n}_{\geq 0}=\ZZ^{(\NN)}_{\geq 0}$.  The
closure operations are $\mndcl_n=\Mon(\cdot)$ in $R_{n}$ for $n\geq 1$
and $\mndcl_\infty=\Mon(\cdot)$ in $R_{\infty}$, where $\Mon(\cdot)$
denotes the monoid closure.  For $A\subseteq\ZZ^{(\NN)}$ it is defined
as
$ \Mon(A)=\{\sum_{i=1}^km_ia_i\mid k\in\NN,\ a_i\in A,\
m_i\in\ZZ_{\geq 0}\}$.
Now the theory develops in complete analogy to the case of cones, so
we omit the details.  One first shows that the system of monoid
closures $(\cmnd,\mndcl_\infty)$ is consistent.  An easy modification
of Example~\ref{ex:cone-inconsistency} shows that
$(\cmnd,\mndcl_\infty)$ is not consistent if one works in the ambient
chain~$(\ZZ^{n})_{n\ge1}$.  As in Lemma~\ref{l:cone-compatibility},
$\cmnd$ is compatible with $\Pi$ for $\Pi = \Sym(\infty)$ and
$\Pi = \Inc$.  Again, if $\M=(M_{n})_{n\geq 1}$ is an increasing chain
of monoids with $M_{n} \subseteq \ZZ^{n}$ for all $n\ge1$, then the
limit $M_\infty$ is a monoid in~$R_{\infty}$.

With these ingredients
one finds the monoid versions of Theorems~\ref{thm:setFG} and
~\ref{t:local-global-Inc}.

\begin{corollary}
\label{c:globallocalmonoids}
Let $\M=(M_{n})_{n\geq 1}$ be a $\Sym$-invariant chain of monoids
$M_{n} \subseteq \ZZ^{n}_{\geq 0}$ with limit~$M_\infty$.  Then the
following statements are equivalent:
\begin{enumerate}[(a)]
\item $\M$ stabilizes and is eventually finitely generated;
\item
  There exists an $r\in\NN$ such that for all $n\ge r$ the following hold:
  \begin{enumerate}[(i)]
  \item $M_\infty \cap \ZZ^{n}_{\geq 0} = M_{n}$,
  \item $M_{n}$ is finitely generated by elements of support size at
    most $r$;
  \end{enumerate}
\item $M_\infty$ is $\Sym$-equivariantly finitely generated.
\end{enumerate}
\end{corollary}

\begin{corollary}
\label{c:monoid-local-global-Inc}
Let $\M=(M_{n})_{n\geq 1}$ be an $\Inc$-invariant chain of monoids
$M_{n} \subseteq \ZZ^{n}_{\geq 0}$ with limit $M_\infty$.  Consider
the following statements:
\begin{enumerate}[(a)]
\item $\M$ stabilizes and is eventually finitely generated.
\item $M_\infty$ is $\Inc$-equivariantly finitely generated.
\end{enumerate}
Then (a) implies (b) and if $\M$ is eventually saturated, then (a) and
(b) are equivalent.
\end{corollary}

Various possible relaxations of the above results are false.  We
briefly discuss two.

\begin{example}
  \label{e:finGenLocal}
  Let $M_{1}=\{0\}$, $M_{2} = \mndcl_2( e_1+ke_2 \mid k\in \NN)$ and
  $M_{n} = \mndcl_n({\Sym}_{2,n}(M_{2}))$ for $n\geq 3$.  Then
  $\M=(M_n)_{n\geq 1}$ is a $\Sym$-invariant chain of monoids.  One
  checks that $M_n$ is not finitely generated for $n\ge2$ and
  $M_\infty$ is not $\Sym$-equivariantly finitely generated.  By
  definition, $\M$ stabilizes.  Thus, in
  Corollary~\ref{c:globallocalmonoids}(a) the second assumption cannot
  be omitted.
\end{example}

\begin{example}
  \label{e:NoHillerSullivant}
  Let $M_1=\{0\}$, $M_{2} = \mndcl_2(e_{1} + 2e_{2})$ and
  $M_{n} = \mndcl_n( \Inc_{n-1,n}(M_{n-1}) \cup \{e_{1} + ne_{n}\} )$
  for $n\geq 3$.  This $\Inc$-invariant chain of finitely generated
  monoids does not stabilize and $M_\infty$ is not
  $\Inc$-equivariantly finitely generated, since for $i<k\leq j$, the
  generator $e_{i} + ke_{j}$ is irredundant.  Thus, the first
  assumption in Corollary~\ref{c:monoid-local-global-Inc}(a) is
  necessary for~(a)$\Rightarrow$(b).
\end{example}

\begin{remark}
  As for the cone and monoid situations studied in this section, one
  can also consider invariant chains of polytopes in
  $\RR_{\geq 0}^{(\NN)}$. It is easy to derive the polytope version of
  Theorems~\ref{thm:setFG} and~\ref{t:local-global-Inc} that are
  analogous to what we have done for cones and monoids.
\end{remark}

\section{Chains of ideals}
\label{sec:ideals}

We apply Theorem~\ref{thm:setFG} and strengthen
Theorem~\ref{t:local-global-Inc} when the equivariant chain consists
of ideals in polynomial rings.  These results are known in the
literature.  However, the proof in~\cite{hillar2012finite} of
Theorem~\ref{thm:Inc-ideal-stable} contains a gap.  We fill this gap.

Consider the ambient chain $\R=(R_n)_{n\ge1}$ of polynomial rings
$R_n=K[X_{n}]$.  Then $R_\infty=K[x_{i,j}\mid i\in [c],\ j\in \NN]$.
For $n\in\NN\cup\{\infty\}$ let $\langle\cdot\rangle_{n}$ denote the
ideal closure operation in~$R_n$, that is, if $A\subseteq R_n$ then
$ \langle A\rangle_{n}=\{\sum_{i=1}^kf_ia_i\mid k\in\NN,\ f_i\in R_n,
\ a_i\in A\}$.  One checks that these closure operations satisfy the
assumptions of Theorems~\ref{thm:setFG} and~\ref{t:local-global-Inc}:
$\langle\cdot\rangle_\infty$ is consistent with the chain
$(\langle\cdot\rangle_{n})_{n\ge1}$; the chain
$(\langle\cdot\rangle_{n})_{n\ge1}$ is compatible with both $\Sym$ and
$\Inc$; and for any chain $\Icc=(I_n)_{n\ge 1}$ of increasing ideals
$I_n\subseteq R_n$ the limit set $I_\infty=\bigcup_{n\ge1} I_n$ is an
ideal in $R_\infty$.

Hilbert's basis theorem (see, e.g. \cite[Theorem~1.2]{Eis95}) says
that $R_n$ is a Noetherian ring for all $n\in\NN$, that is, every
ideal in $R_n$ is finitely generated.  The ring $R_\infty$ is not
Noetherian, but it is $\Pi$-Noetherian for $\Pi=\Sym$ and $\Pi=\Inc$,
in the sense that every $\Pi$-invariant ideal in $R_\infty$ is
$\Pi$-equivariantly finitely generated.
This result was first proved
by Cohen~\cite[Proposition~2]{Co67}, \cite[Theorem~7]{Co87} and later
rediscovered by Aschenbrenner and Hillar~\cite[Theorem~1.1]{AH07} and
Hillar and Sullivant~\cite[Theorem~3.1,
Corollary~3.5]{hillar2012finite}.  Moreover,
\cite[Theorem~3.1]{hillar2012finite} shows that every $\Pi$-invariant
ideal in $R_\infty$ has a finite $\Pi$-Gr\"{o}bner basis with respect
to the lexicographic order $\preccurlyeq$ on $R_\infty$ induced by
\[
  x_{i,j}\preccurlyeq x_{k,l}\Leftrightarrow j<l \ \text{ or }\ j=l \
  \text{ and }\ i\le k.
\]
Here a subset $G$ of a $\Pi$-invariant ideal $I\subseteq R_\infty$ is
a \emph{$\Pi$-Gr\"{o}bner basis} for $I$ with respect to
$\preccurlyeq$, if for any $f\in I$ there exist $g\in G$ and
$\pi\in\Pi$ such that $\ini_\preccurlyeq(f)$ is divisible by
$\ini_\preccurlyeq(\pi(g))$.

The above results show that chains of ideals are better behaved than
chains of cones and monoids.  In particular, the statements of
Theorems~\ref{thm:setFG}(c) and \ref{t:local-global-Inc}(b) on
the equivariant finite generation of the limit ideal are \emph{always
  true} for any $\Sym$- or $\Inc$-invariant chain of ideals.  Thus, for $\Sym$-invariant chains of ideals Theorem~\ref{thm:setFG} immediately yields the following consequence, the
first two statements of which are
\cite[Corollary~3.7]{hillar2012finite} (see also
\cite[Theorem~4.7]{AH07}).

\begin{corollary}
 \label{c:Sym-ideal-stable}
 Let $\Icc=(I_n)_{n\ge 1}$ be a $\Sym$-invariant chain of ideals
 $I_n\subseteq R_n$.  Then
 \begin{enumerate}
 \item $\Icc$ stabilizes.
 \item There exists an $r\in\NN$ such that for all $n\ge1$, $I_n$ is
   finitely generated by elements of support size at most~$r$.
 \item $\Icc$ is eventually saturated.
 \end{enumerate}
\end{corollary}

For $\Inc$-invariant chains, using that
Theorem~\ref{t:local-global-Inc}(b) always holds, one obtains:

\begin{corollary}
 \label{c:saturated-Inc}
 Let $\Icc=(I_n)_{n\ge 1}$ be an $\Inc$-invariant chain of ideals with
 $I_n\subseteq R_n$ for all $n\ge1$. If $\Icc$ is eventually
 saturated, then it stabilizes.
\end{corollary}

Surprisingly, compared to cones and monoids, this result can be much
improved.

\begin{theorem}
  \label{thm:Inc-ideal-stable}
  Every $\Inc$-invariant chain of ideals stabilizes.
\end{theorem}

This result is the second part of
\cite[Theorem~3.6]{hillar2012finite}, but its proof contains a gap as
we now explain. In order to show that an $\Inc$-invariant chain of
ideals $\Icc=(I_n)_{n\ge 1}$ stabilizes, the idea in
\cite{hillar2012finite} is to consider the global ideal
$I_\infty=\bigcup_{n\ge1} I_n$, show that this ideal has a finite
$\Inc$-Gr\"{o}bner basis based on Higman's lemma (see
\cite[Theorem~3.1]{hillar2012finite}), and then apply
\cite[Lemma~2.18]{hillar2012finite} to get the desired conclusion (see
\cite[Theorem~2.19]{hillar2012finite}).  However, the problem here is
that \cite[Lemma~2.18]{hillar2012finite} is applicable only to
saturated chains, and so this argument works only for such chains.

To fix the gap, one can still follow the idea of using Higman's lemma
as in \cite{hillar2012finite}, but instead of passing to the global
ideal $I_\infty$, one needs to work directly on the chain $\Icc$. The
following proof was suggested by an anonymous referee and is much
simpler than our original one.

\begin{proof}[Proof of Theorem~\ref{thm:Inc-ideal-stable}]
  Let $\Icc=(I_n)_{n\ge 1}$ be an arbitrary $\Inc$-invariant chain of
  ideals. Using the lexicographic order and a Gr\"obner basis argument, we may assume that the chain
  $\Icc$ consists of monomial ideals (see,
  e.g.\ \cite[Lemma~7.1]{nagel2017equivariant}).
  For every $n\in\NN$
  we have a bijection between the set $\Mo(R_n)$ of monomials of $R_n$
  and $(\NN^c)^n$, in which each $u\in\Mo(R_n)$ is mapped to its
  exponent vector. Thus, there is a bijection between the disjoint
  unions
  \[
    \Mo(\R)\defas\biguplus_{n\ge 1}\Mo(R_n)
    \quad\text{and}\quad
    (\NN^c)^*\defas\biguplus_{n\ge 1}(\NN^c)^n.
  \]
  Endowing $\Mo(\R)$ and $(\NN^c)^*$ with the Higman partial order
  (see, e.g.\ \cite[Definition~3.2]{hillar2012finite}) the above
  bijection is in fact a poset isomorphism. By Higman's
  lemma~\cite[Theorem~4.3]{Hig}, the Higman order is a
  well-partial-order on~$(\NN^c)^*$. Hence, it is so on $\Mo(\R)$ as
  well, and the chain $\Icc$ stabilizes.
\end{proof}

\begin{remark}
  In the context of FI- and OI-modules, a similar idea to the one used
  in the proof of Theorem~\ref{thm:Inc-ideal-stable} was employed in
  \cite{nagel2019FIOI} to show that certain FI- and OI-modules have
  finite Gr\"{o}bner bases (see, in particular,
  \cite[Propositions~5.3, 6.2, Theorem~6.14]{nagel2019FIOI}).
\end{remark}

We conclude this section with a discussion of the stability index of
$\Inc$-invariant chains.

\begin{example}\label{ex:nobound}
  Although any non-saturated $\Inc$-invariant chain of ideals
  stabilizes by Theorem~\ref{thm:Inc-ideal-stable}, its stability
  index can behave very badly, unlike that of a saturated chain.  Let
  $\Icc=(I_n)_{n\ge 1}$ be an $\Inc$-invariant chain of ideals and let
  $G$ be a finite $\Inc$-Gr\"{o}bner basis for the limit ideal
  $I_\infty$ with respect to the lexicographic order
  $\preccurlyeq$. If $\Icc$ is saturated and $G\subseteq R_r$ for some
  $r\in\NN$, then $\ind(\Icc) \le r$ by
  \cite[Lemma~2.18]{hillar2012finite}.  However, such a
  bound does not hold if $\Icc$ is not saturated. Indeed, let $m\in\NN$ and
  consider the chain $\Icc=(I_n)_{n\ge 1}$ with
 \[
  I_n=\begin{cases}
       \langle x_1^2,x_2^2,\dots,x_n^2\rangle_n&\text{if }\ n\le m-1,\\
       \langle x_1,x_2^2,\dots,x_m^2\rangle_n&\text{if }\ n= m,\\
       \langle \Inc_{m, n}(I_{m})\rangle_n&\text{if }\ n\ge m+1.
      \end{cases}
 \]
 Then $\ind(\Icc)=m$, while $G=\{x_1\}\subseteq R_1$ is an
 $\Inc$-Gr\"{o}bner basis for the limit ideal
 $ I_\infty=\bigcup_{n\ge 1}I_n =\langle
 x_1,x_2,x_3,\dots\rangle_\infty =\langle \Inc(x_1)\rangle_\infty$.
 This shows that there is no bound for $\ind(\Icc)$ in terms of the
 smallest index $r$ with $G\subseteq R_r$.
\end{example}

For non-saturated chains one may try to replace the condition
``$G\subseteq R_r$'' by ``$G\subseteq I_s$'' and ask for a bound for
$\ind(\Icc)$ in terms of the smallest index $s$ with $G\subseteq I_s$.
This is a reasonable question because in Example~\ref{ex:nobound},
$G=\{x_1\}\subseteq I_m$ and $s=m=\ind(\Icc)$.  Unfortunately, this also
does not work.  To build an example, note that in the chain $\Icc$
above one has
\[
 I_{m+j}=\langle \Inc_{m, m+j}(I_{m})\rangle_{m+j}
 =\langle x_1,\dots,x_{j+1},x_{j+2}^2,\dots,x_{m+j}^2\rangle_{m+j}
 \quad\text{for }\ j\ge 0.
\]
Now take $l\in\NN$ and consider the chain $\Icc'=(I_n')_{n\ge 1}$ with
\[
  I_n'=\begin{cases}
    I_n&\text{if }\ n\le m+l-1,\\
    I_{m+l}+\langle x_{l+2}\rangle_{n}&\text{if }\ n= m+l,\\
    \langle \Inc_{m+l, n}(I'_{m+l})\rangle_{n}&\text{if }\ n\ge m+l+1.
  \end{cases}
\]
It is clear that
$I'_\infty=\bigcup_{n\ge 1}I'_n =\langle
x_1,x_2,x_3,\dots\rangle_\infty =I_\infty$.  Thus,
$G'=G=\{x_1\}\subseteq I'_m$ is an $\Inc$-Gr\"{o}bner basis for
$I'$. However, $\ind(\Icc')=m+l$ is not bounded by $m$ as $l$ is
arbitrary.

{\bf Acknowledgment.} We would like to thank the anonymous referees
for their insightful and constructive suggestions.

 \bigskip \medskip

 \noindent
 \footnotesize {\bf Authors' addresses:}

 \smallskip

 \noindent Thomas Kahle, OvGU Magdeburg, Germany,
 {\tt thomas.kahle@ovgu.de}

 \noindent Dinh Van Le, Universit\"at Osnabr\"uck, Osnabr\"uck, Germany,
 {\tt dlevan@uos.de}

 \noindent Tim R\"{o}mer, Universit\"at Osnabr\"uck, Osnabr\"uck, Germany,
 {\tt troemer@uos.de}


\begin{thebibliography}{99}

\bibitem{AH07}
M. Aschenbrenner and C.J. Hillar,
\emph{Finite generation of symmetric ideals}.
Trans. Amer. Math. Soc. {\bf 359} (2007), no. 11, 5171--5192.

\bibitem{symmetric-poly}
  D.~Bremner, M.D.~Sikiric, and A.~Sch\"{u}rmann,
  \emph{Polyhedral representation conversion up to symmetries}.
In: \emph{CRM Proceedings and Lecture Notes} {\bf 48} (2009), 45--72.

\bibitem{bruns2009book}
W. Bruns and J. Gubeladze,
\emph{Polytopes, rings, and $K$-theory}.
Springer Monographs in Mathematics, Springer, 2009.

\bibitem{ChElFa14}
T.~Church, J.S.~Ellenberg, B.~Farb, and R.~Nagpal,
\emph{FI-modules over Noetherian rings}.
Geom. Topol. {\bf 18} (2014), no. 5, 2951--2984.

\bibitem{FI}
T.~Church, J.S.~Ellenberg, and B.~Farb,
\emph{FI-modules and stability for representations of symmetric groups}.
Duke Math. J.  {\bf 164} (2015), no. 9, 1833--1910.

\bibitem{ChEl17}
T.~Church and J.S.~Ellenberg,
\emph{Homology of FI-modules}.
Geom. Topol. {\bf 21} (2017), no. 4, 2373--2418.

\bibitem{Co67}
D.E. Cohen,
\emph{On the laws of a metabelian variety}.
J. Algebra {\bf 5} (1967), 267--273.

\bibitem{Co87}
D.E. Cohen,
\emph{Closure relations, Buchberger's algorithm, and polynomials in infinitely many variables}.
In: \emph{Computation theory and logic}, 78--87, Lecture Notes in Comput. Sci., {\bf 270}, Springer, Berlin, 1987.

\bibitem{Dr14}
J. Draisma,
\emph{Noetherianity up to symmetry}.
In: {Combinatorial algebraic geometry},
Lecture Notes in Mathematics {\bf 2108}, pp. 33--61, Springer, 2014.

\bibitem{Eis95}
D. Eisenbud,
\emph{Commutative algebra. With a view toward algebraic geometry}. Graduate Texts in Mathematics {\bf  150}, Springer, 1995.

\bibitem{NG18}
  S.~G{\"u}nt{\"u}rk{\"u}n and U.~Nagel
  \emph{Equivariant Hilbert series of monomial orbits}.
  Proc. AMS.
  {\bf 146} (2018) no. 6, 2381--2393.

\bibitem{Hig}
G. Higman,
\emph{Orderings by divisibility in abstract algebras}.
Proc. London Math. Soc. (3) {\bf 2} (1952), 326--336.


\bibitem{hillar2012finite}
C.J. Hillar and S. Sullivant,
\emph{Finite Gr\"{o}bner bases in infinite dimensional polynomial rings and applications}.
Adv. Math. {\bf 229} (2012), no. 1, 1--25.

\bibitem{kahle2014equivariant}
T. Kahle, R. Krone, and A. Leykin,
\emph{Equivariant lattice generators and Markov bases}.
In: \emph{ISSAC 2014--Proceedings of the 39th International Symposium on Symbolic and Algebraic Computation}, 264--271, ACM, New York, 2014.

\bibitem{LNNR}
D.V. Le, U. Nagel, H.D. Nguyen, and T. R\"{o}mer,
\emph{Castelnuovo-Mumford regularity up to symmetry}.
Int. Math. Res. Not. {\bf 14} (2021), 11010--11049.

\bibitem{LNNR2}
D.V. Le, U. Nagel, H.D. Nguyen, and T. R\"{o}mer,
\emph{Codimension and projective dimension up to symmetry}.
Math. Nachr. {\bf 293} (2) (2020), 346--362.

\bibitem{le2020kruskal}
D.V. Le and T. R\"{o}mer,
\emph{A Kruskal-Katona type theorem and applications}.
Discrete Math. {\bf 343} (5) (2020).

\bibitem{LR}
 D.V. Le and T. R\"{o}mer,
\emph{Theorems of Carath\'{e}odory, Minkowski-Weyl, and Gordan
up to symmetry}.
Preprint, 2021, available at
\href{https://arxiv.org/abs/2110.10657}{arXiv:2110.10657}.

\bibitem{moore1910}
E.H. Moore,
\emph{Introduction to a form of general analysis}.
The New Haven Mathematical Colloquium, Yale University Press, New Haven, 1910.

\bibitem{nagel2017equivariant}
U. Nagel and T. R\"{o}mer,
\emph{Equivariant Hilbert series in non-Noetherian polynomial rings}.
J. Algebra {\bf 486} (2017), 204--245.

\bibitem{nagel2019FIOI}
U. Nagel and T. R\"{o}mer,
\emph{{\rm FI}- and {\rm OI}-modules with varying coefficients}.
J. Algebra {\bf 535} (2019), 286--322.

\bibitem{GL-sam-snowden}
  S.~Sam and A.~Snowden,
  \emph{$GL$-equivariant modules over polynomial rings in infinitely
    many variables}.
Trans. Amer. Math. Soc. {\bf 368} (2016), no.~2, 1097--1158.

\bibitem{SaSn17}
S.~Sam and A.~Snowden,
\emph{Gr\"obner methods for representations of combinatorial categories}.
J. Amer. Math. Soc. {\bf 30} (2017), no. 1, 159--203.

\bibitem{ziegler1995book}
G.M. Ziegler,
\emph{Lectures on Polytopes}. Graduate Texts in Mathematics {\bf 152}, Springer, 1995.

\end{thebibliography}
\end{document}